\documentclass[11pt]{amsart}
\usepackage[top=1.5in, bottom=1.5in, left=1.4in, right=1.4in]{geometry}                
\usepackage{color}
\usepackage{graphicx}
\usepackage{amssymb}
\usepackage{epstopdf}
\usepackage{amsthm}
\usepackage{hyperref}
\newtheorem{theorem}{Theorem}[section]
\newtheorem{lemma}[theorem]{Lemma}
\newtheorem{definition}[theorem]{Definition}
\newtheorem{corollary}[theorem]{Corollary}
\newtheorem{algorithm}[theorem]{Algorithm}

\title{Riemann-Roch Theory for Graph Orientations}
\author{Spencer Backman}

\keywords{Chip-firing, partial graph orientation, cycle-cocycle reversal system, Dhar's algorithm, Riemann-Roch theorem for graphs, rank-determining set, max-flow min-cut theorem.}

\begin{document}

\bibliographystyle{plain}

\begin{abstract}
We develop a new framework for investigating linear equivalence of
divisors on graphs using a generalization of Gioan's cycle--cocycle
reversal system for partial orientations. An oriented version of
Dhar's burning algorithm is introduced and employed in the study of
acyclicity for partial orientations. We then show that the
Baker--Norine rank of a partially orientable divisor is one less than
the minimum number of directed paths which need to be reversed in the
generalized cycle--cocycle reversal system to produce an acyclic partial
orientation. These results are applied in providing new proofs of the
Riemann--Roch theorem for graphs as well as Luo's topological
characterization of rank-determining sets. We prove that the max-flow
min-cut theorem is equivalent to the Euler characteristic description
of orientable divisors and extend this characterization to the setting
of partial orientations. Furthermore, we demonstrate that
$Pic^{g-1}(G)$ is canonically isomorphic as a $Pic^{0}(G)$-torsor to
the equivalence classes of full orientations in the cycle--cocycle
reversal system acted on by directed path reversals. Efficient
algorithms for computing break divisors and constructing partial
orientations are presented.

\end{abstract}

\maketitle

\tableofcontents

\section{Introduction}

Baker and Norine \cite{baker2007riemann} introduced a combinatorial
Riemann--Roch theorem for graphs analogous to the classical statement
for Riemann surfaces. For proving the theorem, they employed
chip-firing, a deceptively simple game on graphs with connections to
various areas of mathematics.
Given a graph $G$, we define a \textit{configuration of chips} $D$ on
$G$ as a function from the vertices to the integers. A vertex $v$
\textit{fires} by sending a chip to each of its neighbors, losing its
degree number of chips in the process. If we take $D$ to be a vector,
firing the vertex $v_{i}$ precisely corresponds to subtracting the
$i$th column of the Laplacian matrix from $D$. In this way we may view
chip-firing as a combinatorial language for describing the integer
translates of the lattice generated by the columns of the Laplacian
matrix, e.g.
\cite{amini2010riemann,nagnibeda1997roland}.

Reinterpreting chip configurations as \textit{divisors}, we say that
two divisors are \textit{linearly equivalent} if one can be obtained
from the other by a sequence of chip-firing moves, and a divisor is
\textit{effective} if each vertex has a nonnegative number of chips.
Baker and Norine define the \textit{rank} of a divisor, denoted
$r(D)$, to be one less than the minimum number of chips which need to
be removed so that $D$ is no longer equivalent to an effective divisor.
Defining the \textit{canonical divisor} $K$ to have values $K(v) =
\mathrm{deg}(v)-2$, the \textit{genus} of $G$ to be $g =
|E(G)|-|V(G)|+1$, and the \textit{degree} $\mathrm{deg}(D)$ of a divisor $D$ to be the total number of chips in $D$, they prove the Riemann--Roch formula for graphs:

\begin{theorem}[Baker--Norine \cite{baker2007riemann}]
\begin{eqnarray*}r(D)-r(K-D) = \mathrm{deg}(D) -g+1.
\end{eqnarray*}
\end{theorem}

Baker and Norine's proof depends in a crucial way on the theory of
\textit{$q$-reduced divisors}, known elsewhere as \textit{$G$-parking
functions} or \textit{superstable configurations} \cite{cori2002polynomial,postnikov2004trees}. A divisor $D$ is said to be
\textit{$q$-reduced} if $(i) \, D(v) \geq0$ for all $v \neq q$, and
$(ii)$ for any non-empty subset $A \subset V(G) \setminus\{q\}$,
firing the set $A$ causes some vertex in $A$ to go into debt, i.e., to
have a negative number of chips. They show that every divisor $D$ is
linearly equivalent to a unique $q$-reduced divisor $D'$, and $r(D)
\geq0$ if and only if $D'$ is effective. We note that $q$-reduced
divisors are dual, in a precise sense, to the \textit{recurrent
configurations} (also known as \textit{$q$-critical configurations}),
which play a prominent role in the abelian sandpile model
\cite[Lemma~5.6]{baker2007riemann}

There is a second story, which runs parallel to that of chip-firing,
describing certain constrained reorientations of graphs first
introduced by Mosesian \cite{mosesian1972strongly} in the context of
Hasse diagrams for posets. Given an acyclic orientation of a graph
$\mathcal{O}
$ and a sink vertex $q$, we can perform a \textit{sink reversal},
reorienting all of the edges incident to $q$. This operation is
directly connected to the theory of chip-firing: we can associate to
$\mathcal{O}$ a divisor $D_{\mathcal{O}}$ with entries $D_{\mathcal
{O}} (v) = \mathrm{indeg}_{\mathcal{O}
}(v) -1$, and performing a sink reversal at $v_{i}$ we obtain the
orientation $\mathcal{O}'$ with associated divisor $D_{\mathcal{O}'}$
given by the
firing of $v_{i}$. Mosesian observed that, provided an acyclic
orientation $\mathcal{O}$ and a vertex $q$, there exists a unique acyclic
orientation $\mathcal{O}'$ having $q$ as the unique sink, which is obtained
from $\mathcal{O}$ by sink reversals. The divisors associated to (the reverse
of) these $q$-rooted acyclic orientations are the maximal noneffective
$q$-reduced divisors. This connection between acyclic orientations and
chip-firing dates back at least to Bj\"{o}rner, Lov\'{a}sz, and Shor's
seminal paper on the topic \cite{bjorner1991chip}, and has been
utilized in recent proofs of the Riemann--Roch
formula \cite{amini2012linear,cori2013riemann,mikhalkin465tropical}.

Gioan \cite{gioan2007enumerating} generalized this setup to arbitrary
(not necessarily acyclic) orientations by introducing the \textit
{cocycle reversal}, wherein all of the edges in a consistently oriented
cut can be reversed, and a \textit{cycle reversal}, in which the edges
in a consistently oriented cycle can be reversed. Using these two
operations, he defined the \textit{cycle--cocycle reversal system} as
the collection of full orientations modulo cycle and cocycle reversals,
and proved that the number of equivalence classes in this system is
equal to the number of spanning trees of the underlying graph.
He also showed that each orientation is equivalent in the cocycle
reversal system to a unique \textit{$q$-connected orientation}. These
are the orientations in which every vertex is reachable from $q$ by a
directed path. Gioan and Las Vergnas \cite{gioan2005activity}, and
Bernardi \cite{bernardi2008tutte}, combined these results, presenting
an explicit bijections between the $q$-connected orientations with a
standardized choice of the orientation's cyclic part and spanning trees
of a graph. Bernardi's bijection is determined by a choice of
``combinatorial map'', which is essentially a combinatorial embedding of
a graph in a surface. Recently, An, Baker, Kuperberg, and Shokrieh \cite{an2013canonical} showed that the divisors associated to the
$q$-connected orientations are precisely the break divisors of
Mikhalkin and Zharkov \cite{mikhalkin465tropical} offset by a chip at
$q$. They then applied this observation to give a tropical ``volume
proof'' of Kirchoff's matrix-tree theorem via a canonical polyhedral
decomposition of $\mathrm{Pic}^{g}(G)$, the collection of divisors of
degree $g$ modulo linear equivalence.

A limitation of the orientation-based perspective is that the divisor
associated to an orientation will always have degree $g-1$. In this
work, we introduce a generalization of the cycle--cocycle reversal
system for investigating \textit{partial orientations}, thus allowing
for a discussion of divisors with degrees less than $g-1$.
The \textit{generalized cycle--cocycle reversal system} is defined by
the introduction of \textit{edge pivots}, whereby an edge $(u,v)$
oriented towards $v$ is unoriented and an unoriented edge $(w,v)$ is
oriented towards $v$ (see Fig.~\ref{pivotalcycle-cocyclenoshadow}).
Note that edge pivots, as with cycle reversals, leave the divisor
associated to a partial orientation unchanged. We demonstrate that this
additional operation is dynamic enough to allow for a characterization
of linear equivalence.
%

\begin{figure}
\includegraphics[scale=.6]{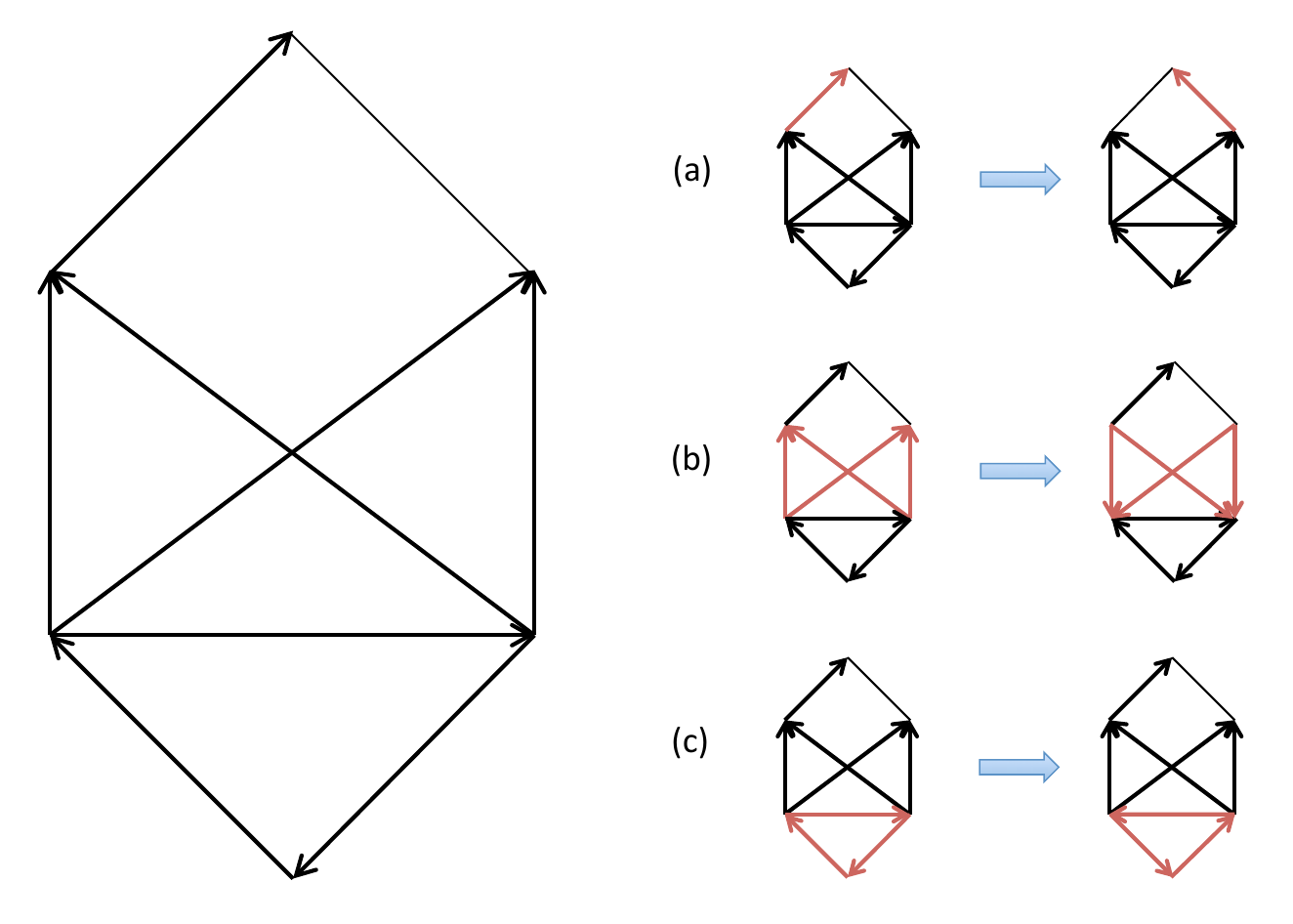}
\caption{A partial orientation with (a) an edge pivot, (b) a cocycle reversal,
and (c) a cycle reversal.}
\label{pivotalcycle-cocyclenoshadow}
\end{figure}

\begin{theorem}
Two partial orientations are equivalent in the generalized
cycle--cocycle reversal system if and only if their associated divisors
are linearly equivalent.
\end{theorem}

Moreover, we use edge pivots and cocycle reversals to show that a
divisor with degree at most $ g-1$ is linearly equivalent to a divisor
associated to a partial orientation or it is linearly equivalent a
divisor dominated by a divisor associated to an acyclic partial
orientation. These results allow us to reduce the study of linear
equivalence of divisors of degree at most $g-1$ on graphs to the study
of partial orientations.

Dhar's burning algorithm is one of the key tools in the study of
chip-firing. Originally discovered in the context of the abelian
sandpile model, Dhar's algorithm provides a quadratic-time test for
determining whether a given configuration is $q$-reduced. There are
variants of Dhar's algorithm which produce bijections between
$q$-reduced divisors and spanning trees, some of which respect
important tree statistics such as external activity \cite{cori2003sand} or tree inversion number \cite{perkinson2016g}. In the
work of Baker and Norine, this algorithm was implicitly employed in the
proof of their core lemma RR1, which states that if a divisor has
negative rank then it is dominated by a divisor of degree $g-1$ which
also has negative rank. We present an oriented version of Dhar's
algorithm whose iterated application provides a method for determining
whether a partial orientation is equivalent in the generalized cocycle
reversal system to an acyclic partial orientation or a sourceless
partial orientation. We combine these results to obtain the following theorem.\looseness=1

\begin{theorem}
Let $D$ be a divisor with $\mathrm{deg}(D)\leq g-1$, then

(i)
 $r(D) = -1$ if and only if $D \sim D' \leq D_{\mathcal{O}}$ with
$\mathcal{O}$ an
acyclic partial orientation.

(ii)
 $r(D) \geq0$ if and only if $D \sim D_{\mathcal{O}}$ with
$\mathcal{O}$ a
sourceless partial orientation.
\end{theorem}

This implies that for understanding whether the rank of a divisor is
negative or nonnegative, it suffices to investigate partial
orientations. We introduce $q$-connected partial orientations and
use them to prove the following explicit description of ranks of
divisors associated to partial orientations.

\begin{theorem}
The Baker--Norine rank of a divisor associated to a partial orientation
is one less than the minimum number of directed paths which need to be
reversed in the generalized cycle-cocycle reversal system to produce an
acyclic partial orientation.
\end{theorem}

We apply these results in providing a new proof of the
Riemann--Roch theorem for graphs. For this, we employ a variant of Baker
and Norine's formal reduction involving strengthened versions of RR1
and RR2. The Riemann--Roch theorem was extended to metric graphs and
tropical curves by Gathmann and Kerber \cite{gathmann2008riemann}, and
Mikhalkin and Zharkov \cite{mikhalkin465tropical}. We are currently
writing an extension of the results from this paper to the setting of
metric graphs.

Luo \cite{luo2011rank} investigated the notion of a \textit
{rank-determining set} for a metric graph $\Gamma$, a collection $A$
of points such that the rank of any divisor can be computed by removing
chips only from points in $A$. A \textit{special open set} $\mathcal
{U}$ is a
nonempty, connected, open subset of $\Gamma$ such that every connected
component $X$ of $\Gamma\setminus\mathcal{U}$ has a boundary point
$p$ with
$\mathrm{outdeg}_{X}(p)\geq2$. We apply acyclic orientations and path
reversals to provide a new proof of Luo's topological characterization
of rank-determining sets as those which intersect every special open set.

We discuss a close relationship between network flows and divisor
theory. We demonstrate that the max-flow min-cut theorem is logically
equivalent to the Euler characteristic description of orientable
divisors \cite{an2013canonical}. A polynomial-time method for
computing break divisors is provided, combining the observation
(originally due to Felsner \cite{felsner2004lattice}) that max-flow
min-cut can be used to construct orientations, with An, Baker,
Kuperberg, and Shokrieh's characterization of break divisors as the
divisors associated to $q$-connected orientations offset by a chip at
$q$. Motivated by these connections with max-flow min-cut, we prove the
following statement.

\begin{theorem}
$Pic^{g-1}$ is canonically isomorphic as a $Pic^{0}$-torsor to the set
of full orientations modulo cut and cycle reversals acted on by path reversals.
\end{theorem}

The perspective given by partial orientations is more ``matroidal'' than
the divisor theory of Baker and Norine. In future work, we plan to
extend the ideas from this paper to \textit{partial reorientations of
oriented matroids}.

\section{Notation and terminology}

\textbf{Graphs:} We take $G$ to be a finite loopless undirected connected
multigraph with vertex set $V(G)$ and edge set $E(G)$. The \textit
{degree} of $v$, written $\mathrm{deg}(v)$, is the number of edges
incident to $v$ in $G$. For $X,Y \subset V(G)$, we write $(X,Y)$ for
the set of edges with one end in $X$ and the other in $Y$. Thus
$(X,X^{c})$ is the cut defined by $X$. Given $v \in V(G)$, we write
$\mathrm{outdeg}_{X}(v)$ for the number of edges incident to $v$ leaving
the set $X$. We define the \textit{boundary} of $X$ to be the set of
vertices in $X$ such that $\mathrm{outdeg}_{X}(v)>0$. For $S,T \subset
V(G)$ we say that the \textit{distance} from $S$ to $T$, written
$d(S,T)$ is the minimum over all $s \in S$ and $t \in T$ of the length
of a walk from $s$ to $t$.

A \textit{divisor}, or a \textit{chip configuration}, is a formal sum
of the vertices with integer coefficients. Alternately, a divisor may
be considered as function $D: V(G) \rightarrow\mathbb{Z}$, i.e., an
integral vector. We denote the set of divisors on $G$ by $\mathrm{Div}(G)$. The net number of chips in a divisor $D$ is called the
\textit{degree} of $D$ and is written $\mathrm{deg}(D)$. Given two
divisors $D_{1}$ and $D_{2}$, we write $D_{1} \geq D_{2}$ if $D_{1}(v)
\geq D_{2}(v)$ for all $v \in V(G)$. If $D_{1} \geq D_{2}$ and $D_{1}
\neq D_{2}$, we may write $D_{1} \gneq D_{2}$.   Let $D_{1}$ and $D_{2} $
be the effective divisors with disjoint supports such that $D_{1} -
D_{2} = D$. We write $\mathrm{deg}^{+}(D)$ and $\mathrm{deg}^{-}(D)$ for
$\mathrm{deg}(D_{1})$ and $\mathrm{deg}(D_{2})$, respectively.

We take $\Delta$ to be the \textit{Laplacian} matrix $\Delta= D-A$,
where $D$ is a diagonal matrix with $(i,i)$th entry $\mathrm{deg}(v_{i})$,
and $A$ is the adjacency matrix with $(i,j)$th entry equal to the
number of edges between $v_{i}$ and $v_{j}$. If a vertex $v_{i}$
\textit{fires}, it sends a chip to each of its neighbors, losing its
degree number of chips in the process, and we obtain the new divisor
$D-\Delta e_{i}$, where $e_{i}$ is the $i$th standard basis vector. We
define the firing of a set of vertices to be the firing of each vertex
in that set. Given $A \subset V(G)$, we take $\chi_{A}$ to be the
incidence vector for $A$ so that $D-\Delta\chi_{A}$ is the divisor
obtained by firing the set~$A$. We say that two divisors $D$ and $D'$
are \textit{linearly equivalent}, written $D \sim D'$, if there exists
a sequence of firings bringing $D$ to $D'$, i.e., $D-D'$ is in the
$\mathbb{Z}$-span of the columns of $\Delta$. We define $\mathrm{Pic}^{d}(G)$ to be the set of divisors of degree $d$ modulo linear equivalence.

A vertex $v$ is in debt if $D(v) <0$, and $D$ is \textit{effective} if
no vertex is in debt. The \textit{rank} of a divisor is the quantity
$r(D) = \min_{E\geq0} \mathrm{deg}(E)-1$ such that there exists no $E'
\geq0$ with $D - E \sim E'$. The genus of a graph $g =
|E(G)|-|V(G)|+1$, also known as the cyclomatic number of $G$, is the
rank of the cographic matroid of $G$. The \textit{canonical divisor}
$K$ is the divisor with $i$th entry $K(v_{i})=\mathrm{deg}(v_{i})-2$. A
divisor $D$ is said to be \textit{$q$-reduced} for some $q \in V(G)$
if $(i)$ $D(v) \geq0$ for all $v \neq q$, and $(ii)$ for any set $A
\subset V(G) \setminus\{ q \} $, firing $A$ causes some vertex to be
sent into debt. We take the set of \textit{non-special divisors} to be
$\mathcal{N} = \{ \nu\in\mathrm{Div}(G) : \mathrm{deg}(\nu) = g-1, r(\nu
) = -1\}$. We also define
\begin{eqnarray*}\mathcal{N}_{k} = \{ \nu\in\mathrm{Div}(G) : \mathrm{deg}(\nu) = k, r(\nu) = -1\}
\end{eqnarray*}
so that $\mathcal{N}_{g-1} = \mathcal{N}$.

For a non-empty $S \subset V(G)$, we take $G[S]$ to be the induced
subgraph on $S$ and let $D|_{S}$ be the divisor $D$ restricted to $S$.
We define $\chi(S)$ to be the \textit{topological Euler
characteristic} of $G[S]$, i.e., $|S|-|E(G[S])|$. Given a divisor $D$
and a non-empty subset $S \subset V(G)$, we define
\begin{eqnarray*}[cc]
\chi(S,D) = \mathrm{deg}(D|_{S}) +\chi(S)
\\
\chi(G,D) = \mathrm{min}_{S\subset V(G)} \chi(S,D)
\\
{\bar{\chi}(S,D)} = |E(G[S])| + |(S, S^{c} )| - |S|
- \mathrm{deg}(D|_{S})
\\
{\bar{\chi}}(G,D) = \mathrm{min}_{S\subset V(G)} {\bar
{\chi}(S,D)}.
\end{eqnarray*}

\textbf{Orientations:}
An \textit{orientation} of an edge $e =(u,v) \in E(G)$ is a pairing
$(e,v)$. In this case we say that $e$ is
oriented away from $u$ and oriented towards $v$.  The \textit{tail} of $e$ is $u$ and the \textit{head} of $e$ is $v$.  We draw an oriented
edge, i.e., directed edge, as an arrow pointing from $u$ to~$v$.
A~\textit{partial orientation} $\mathcal{O}$ of a graph is an
orientation of a
subset of the edges, and we say that the remaining edges are \textit
{unoriented}. A~partial orientation is said to be \textit{full}, or
simply an orientation, if each edge in the graph is oriented. A \textit
{directed path} is a path such that the head of each oriented edge is
tail of its successor. For a partial orientation $\mathcal{O}$ and a
set $X
\subset V(G)$, we write ${\bar{X}}_{\mathcal{O}}$ for the set of vertices
reachable from $X$ by a directed path in $\mathcal{O}$, or simply
${\bar{X}}$
when $\mathcal{O}$ is clear from the context.

The indegree of a vertex $v$ in $\mathcal{O}$, written $\mathrm{indeg}_{\mathcal{O}}(v)$
or simply $\mathrm{indeg}(v)$, is the number of edges oriented towards $v$
in $\mathcal{O}$. We associate to each partial orientation, a divisor
$D_{\mathcal{O}}$
with $ D_{\mathcal{O}} (v) = \mathrm{indeg}_{\mathcal{O}}(v)-1$. We say
that a divisor is
\emph{partially orientable}, resp.~\emph{orientable}, if it is the
divisor associated to some partial, resp.~full, orientation. Given a
partially orientable divisor $D$ we denote by $\mathcal{O}_{D}$ any partial
orientation with associated divisor $D$.

An \textit{edge pivot} at a vertex $v$ is an operation on a partial
orientation $\mathcal{O}$ whereby an edge oriented towards $v$ is unoriented
and an unoriented edge incident to $v$ is oriented towards $v$. We say
that a cut (also called a \textit{cocycle}) is \textit{saturated} if
each edge in the cut is oriented. A cut is \textit{consistently
oriented} in $\mathcal{O}$ if the cut is saturated and each edge in
the cut is
oriented in the same direction. We may also refer to this cut as being
\textit{saturated towards $A$} if the cut is consistently oriented
towards $A$. We similarly define a consistently oriented, i.e.
directed, cycle in $\mathcal{O}$. A \textit{cut reversal},
resp.~\textit
{cycle reversal}, in $\mathcal{O}$ is performed by reversing all of
the edges
in a consistently oriented cut, resp.~cycle. The \textit{cycle},
resp.~\textit{cocycle}, resp.~\textit{cycle--cocycle reversal systems}
are the collections of full orientations of a graph modulo cycle,
resp.~cocycle, resp.~cycle and cocycle reversals. The \textit
{generalized cycle}, resp.~\textit{cocycle}, resp.~\textit
{cycle--cocycle reversal systems} are the previous systems extended to
partial orientations by the inclusion of edge pivots. If two partial
orientations $\mathcal{O}$ and $\mathcal{O}'$ are equivalent in the
generalized
cycle--cocycle reversal system, we simply say that they are \textit
{equivalent} and write $\mathcal{O}\sim\mathcal{O}'$.

We say a vertex is a \textit{source} in a partial orientation if it
has no incoming edges. We say that a partial orientation is \textit
{sourceless} if it has no sources, and \textit{acyclic} if it contains
no directed cycles. We denote the set of partial orientations
equivalent to a given partial orientation $\mathcal{O}$ by $[\mathcal
{O}]$. Given a
vertex $q$, a partial orientation is said to be \textit{$q$-connected}
if there exists a directed path from $q$ to every other vertex.

\section{Generalized cycle, cocycle, and cycle--cocycle reversal systems}

In this section we prove Lemma~\ref{generalizedcycle} and Theorem~\ref{generalizedcc}, which generalize results of Gioan \cite[Proposition
4.10 and Corollary 4.13]{gioan2007enumerating} to the setting of
partial orientations. Our {{Theorem~\ref{generalizedcc}} lays the
foundation for the rest of the paper, and states that for divisors
associated to partial orientations, we can understand linear
equivalence as a shadow of the generalized cycle--cocycle reversal
system. We believe that this result may be of independent interest to
those studying chip-firing who are not necessarily interested in
Riemann--Roch theory for graphs.

\begin{lemma}\label{generalizedcycle}
Two partial orientations $\mathcal{O}$ and $\mathcal{O}'$ are
equivalent in the
generalized cycle reversal system if and only if $D_{\mathcal{O}}=
D_{\mathcal{O}'}$.
\end{lemma}

\begin{proof}

Clearly, if $\mathcal{O}$ and $\mathcal{O}'$ are equivalent in the
generalized cycle
reversal system then $D_{\mathcal{O}}= D_{\mathcal{O}'}$. We now
demonstrate the converse.

Suppose there exists some vertex $v$ incident to an edge $e$ which is
oriented towards $v$ in $\mathcal{O}$ and is unoriented in $\mathcal{O}'$. Because
$D_{\mathcal{O}}= D_{\mathcal{O}'}$, there exists another edge $e'$
which is oriented
towards $v$ in $\mathcal{O}'$ such that $e'$ is not also oriented
towards $v$
in $\mathcal{O}$. We can perform an edge pivot in $\mathcal{O}$ so
that $e'$ becomes
unoriented and $e$ is now oriented towards $v$ in both $\mathcal{O}$
and $\mathcal{O}'$. By induction on the number of edges which are oriented in $\mathcal{O}$,
but unoriented in $\mathcal{O}'$, we can assume that no such edge $e$ exists.

We claim that the orientations now differ by cycle reversals. Let $e$
be some edge oriented towards $v$ in $\mathcal{O}$ and away from $v$
in $\mathcal{O}'$.
Because $D_{\mathcal{O}}= D_{\mathcal{O}'}$ there exists another edge
$e'$ which is
oriented away from $v$ in $\mathcal{O}$ and towards $v$ in $\mathcal
{O}'$. We may perform a directed walk along edges in $\mathcal{O}$ which are oriented
oppositely in $\mathcal{O}'$. By the assumption that every edge which is
oriented in $\mathcal{O}$ is also oriented in $\mathcal{O}'$, this
walk must eventually
reach a vertex which has already been visited. This gives a cycle which
is consistently oriented in $\mathcal{O}$ and $\mathcal{O}'$ with opposite
orientations. We can reverse the orientation of this cycle in $\mathcal
{O}$ and
again induct on the number of edges with different orientation in
$\mathcal{O}$
and $\mathcal{O}'$, thus proving the claim.

\end{proof}

When moving from $\mathcal{O}$ to $\mathcal{O}'$ in the proof of
Lemma~\ref{generalizedcycle}, we first perform all necessary edge pivots and then
cycle reversals. If $\mathcal{O}'$ is acyclic then we never need to
perform any
cycle reversals and we obtain the following corollary.

\begin{corollary}\label{lem3.1cor}
Let $\mathcal{O}$ and $\mathcal{O}'$ be partial orientations with
$\mathcal{O}'$ acyclic such
that $D_{\mathcal{O}} = D_{\mathcal{O}'}$, then $\mathcal{O}$ and
$\mathcal{O}'$ are related by a
sequence of edge pivots.
\end{corollary}

We now introduce a nonlocal extension of edge pivots.

\begin{definition}
Given a directed path $P$ from $u$ to $v$ in $G$ in a partial
orientation $\mathcal{O}$, and an unoriented edge $e$ incident to $v$,
we may
perform successive edge pivots along $P$ causing the initial edge of
the path to become unoriented. We call this sequence of edge pivots a
Jacob's ladder cascade (see {Fig.~\ref{jacob'sladder6}}).
\end{definition}

\begin{figure}
\includegraphics[scale=.8]{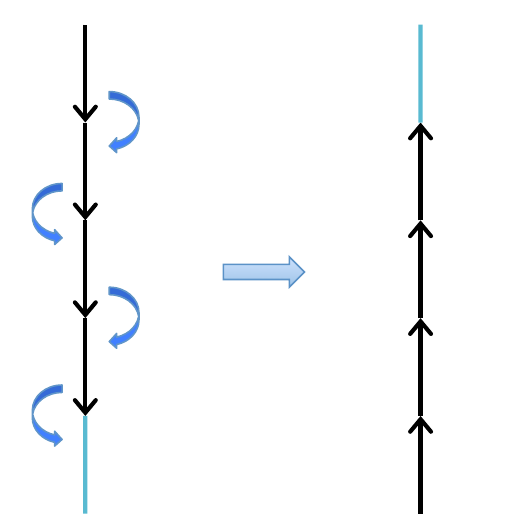}
\caption{A Jacob's ladder cascade.} \label{jacob'sladder6}
\end{figure}

For proving the main result of this section, {Theorem~\ref{generalizedcc}}, we will need the following lemma.

\begin{lemma}\label{bar}
If $\mathcal{O}$ is a partial orientation, then ${\bar{\chi
}}(D_{\mathcal{O}})\geq0$.
\end{lemma}

\begin{proof}
By definition, ${\bar{\chi}}(D_{\mathcal{O}})\geq0$ says that for
every $S
\subset V(G)$, the value ${\bar{\chi}}(S ,D_{\mathcal{O}})\geq0$. This
inequality states that the divisor $D_{\mathcal{O}}$ restricted to the
set $S$
has at most as many chips as can be contributed by oriented edges in the induced subgraph $G[S]$ and the edges in the cut $(S,S^{c})$.
\end{proof}

We note that the induced graph $G[S]$ is fully oriented and all of the
edges in $(S,S^{c})$ are oriented towards $S$ precisely when ${\bar
{\chi}}(S ,D_{\mathcal{O}}) = 0$. Bartels, Mount, and Welsh \cite[Proposition 1]{bartels1997win} proved that the converse of {Lemma~\ref{bar}} also
holds; if a divisor $D$ satisfies ${\bar{\chi}}(D)\geq0$, then $D$
is partially orientable.

\begin{theorem} \label{generalizedcc}
Two partial orientations $\mathcal{O}$ and $\mathcal{O}'$ are
equivalent in the
generalized cycle--cocycle reversal system if and only if $D_{\mathcal
{O}}$ is
linearly equivalent to $D_{\mathcal{O}'}$.
\end{theorem}

\begin{proof}

Clearly, if $\mathcal{O}$ and $\mathcal{O}'$ are equivalent in the generalized
cycle--cocycle reversal system then $D_{\mathcal{O}} \sim D_{\mathcal
{O}'}$. We now
demonstrate the converse.

By {Lemma~\ref{generalizedcycle}}, it suffices to show that there exists
$\mathcal{O}'' \sim\mathcal{O}$ in the generalized cocycle reversal
system such that
$D_{\mathcal{O}''} = D_{\mathcal{O}'}$. Because $D_{\mathcal{O}'}$
and $D_{\mathcal{O}}$ are
chip-firing equivalent, we can write $D_{\mathcal{O}'} = D_{\mathcal
{O}} - \Delta f$
for $f$ some integral vector. The kernel of the Laplacian of an
undirected connected graph is generated by the all 1's vector, thus by
adding the appropriate multiple of the all 1's vector to $f$, we may
assume that $f\geq0$ and there exists some $v \in V(G)$ such that
$f(v) =0$. Let $a$ and $b$ be the minimum and maximum positive values
of $f$ respectively. We take $A = \{ v \in V(G) : f(v) \geq a\} = \mathrm{supp}(f)$ and $ B = \{ v \in V(G) : f(v) = b\}$.

Suppose there exists some set of vertices $C$ with $B \subset C \subset
A$ whose associated cut $(C, C^{c})$ is fully oriented towards $C$. We
can reverse the cut $(C,C^{c})$ to get a new partial orientation
${\tilde{\mathcal{O}}}$ such that $D_{\tilde{\mathcal{O}}}=
D_{\mathcal{O}} - \Delta(f -
\chi_{C})$. The theorem then follows by induction on $\mathrm{deg}(f)$,
hence it suffices to show that we can perform edge pivots to obtain such
a~set~$C$.

We first claim that we may perform edge pivots so that the cut
$(A,A^{c})$ does not contain any edges oriented away from $A$. The
proof of the claim is by induction on the number of edges in
$(A,A^{c})$ oriented away from $A$. Let $e$ be an edge in $(A,A^{c})$
which is oriented towards $A^{c}$. Suppose there exists $p$, a directed
path which begins with $e$ and terminates at some vertex $v$ in $A^{c}$
which is incident to an unoriented edge contained in $G[A^{c}]$. Take
$e'$ to be the last edge in $p$ which belongs to $(A,A^{c})$. Note that
$e'$ is necessarily oriented towards $A^{c}$. We may perform a Jacob's
ladder cascade along the portion of $p$ starting at $e'$, thus causing
$e'$ to become unoriented. By our choice of $e'$, the cascade does not
involve any edges in $(A,A^{c})$ other than $e'$, hence it causes the
number of edges in $(A,A^{c})$ which are oriented towards $A^{c}$ to
decrease by 1, and we may apply induction. Therefore we assume that no
such path $p$ exists, i.e., there is no unoriented edge in $G[A^{c}]$
which is reachable from $A$ by a directed path. Letting $S = \bar
{A}\setminus A$, it follows that every edge in $E(G[S])$ is fully
oriented and $(S, S^{c} \cap A^{c})$ is fully oriented towards $S$. By
assumption $(A,S)$ contains at least one edge oriented towards $S$,
thus we conclude
\begin{eqnarray*}{\bar{\chi}}(S ,D_{\mathcal{O}}) = |E(G[S])| + |(S, S^{c}
\cap A^{c})|+|(A,S)| - |S| - \mathrm{deg}(D_{\mathcal{O}}|_{S})< |(A,S)|.
\end{eqnarray*}
We know that $\mathrm{deg}(D_{\mathcal{O}'}|_{S}) \geq\mathrm{deg}(D_{\mathcal{O}}|_{S}) +
|(A,S)|$ because every vertex in $A$ fires at least once and no vertex
in $S$ fires, therefore ${\bar{\chi}}(S ,D_{\mathcal{O}'})<0$, but
by {Lemma~\ref{bar}} this contradicts the assumption that $D_{O'}$ is a partially
orientable divisor, and proves the claim.

We now assume that none of the edges in $(A,A^{c})$ are oriented
towards $A^{c}$. For each $u \in B$ and $v \notin B$, we have that
$f(u)>f(v)$. This implies that $D_{\mathcal{O}'}(u) \leq D_{\mathcal
{O}}(u)- \mathrm{outdeg}_{B}(u)$, which says that there are at least $\mathrm{outdeg}_{B}(u)$ edges oriented towards $u$. We can perform edge pivots
at vertices on the boundary of $B$ bringing directed edges into the cut
$(B,B^{c})$ which are oriented towards $B$, therefore we assume that no
edge in $(B,B^{c})$ is unoriented. If $(B,B^{c})$ is fully oriented
towards $B$ then we may take $B=C$, so we assume that there exists some
edge $e$ in $(B,B^{c})$ which is oriented towards $B^{c}$. Suppose
there exists $p$, a directed path which begins with $e$ and terminates
at some vertex $v$ in $B^{c}$ which is incident to an unoriented edge
contained in $G[B^{c}]$. By assumption, $(A,A^{c})$ contains no edges
oriented towards $A^{c}$, thus $p$ is contained in $A$. Take $e'$ to be
the last edge in $p$ which belongs to $(B,B^{c})$. Note that $e'$ is
necessarily oriented towards $B^{c}$. We may perform a Jacob's ladder
cascade along the portion of $p$ starting at $e'$ thus causing $e'$ to
become unoriented. By our choice of $e'$, the cascade does not involve
any edges in $(B,B^{c})$ other than $e'$, hence it causes the number of
edges in $(B,B^{c})$ oriented towards $B^{c}$ to decrease by 1 and it
preserves the property that $(A,A^{c})$ contains no edges oriented
towards $A^{c}$. By induction on the number of edges in $(B,B^{c})$
oriented towards $B^{c}$, we assume that no such path $p$ exists. It
follows that $G[{\bar{B}} \setminus B]$ is fully oriented. Moreover,
$(\bar{B},\bar{B}^{c})$ is fully oriented towards ${\bar{B}}$, and
$B \subset\bar{B} \subset A$, hence we may take $C = \bar{B}$, thus
completing the proof.
\end{proof}

\begin{corollary}\label{cocycle}
Let $\mathcal{O}$ and $\mathcal{O}'$ be partial orientations with
$\mathcal{O}'$ acyclic. Then
$\mathcal{O}$ and $\mathcal{O}'$ are equivalent in the generalized
cycle--cocycle
reversal system if and only if they are equivalent in the generalized
cocycle reversal system.
\end{corollary}

\begin{proof}
It is clear that if $\mathcal{O}$ and $\mathcal{O}'$ are equivalent
in the generalized
cocycle reversal system then they are equivalent in the generalized
cycle--cocycle reversal system. For the converse, suppose that $\mathcal
{O}$ and
$\mathcal{O}'$ are equivalent in the generalized cycle--cocycle
reversal system.
By the proof of {Theorem~\ref{generalizedcc}}, $\mathcal{O}$ is
equivalent in
the generalized cocycle reversal system to some partial orientation
$\mathcal{O}
''$ such that $D_{\mathcal{O}''} = D_{\mathcal{O}'}$. Then by
{Corollary~\ref{lem3.1cor}},
$\mathcal{O}''$ is equivalent to $\mathcal{O}'$ in the
generalized cycle
reversal system using only edge pivots as $\mathcal{O}'$ is acyclic.
\end{proof}

In the following section, we will be interested in the question of when
a partially orientable divisor $D_{\mathcal{O}}$ is linearly
equivalent to a
partially orientable divisor $D_{\mathcal{O}'}$ where $\mathcal{O}'$
is acyclic. By
{Corollary~\ref{cocycle}}, it is sufficient to restrict our attention to
the generalized cocycle reversal system.

\section{Oriented Dhar's algorithm} \label{Dhar}

In this section we continue the development of divisor theory for
graphs in the language of the generalized cocycle reversal system
for partial orientations. The main result of this section is {Theorem~\ref{eff1}}, which says that for a divisor $D$ of degree at most $g-1$,
we can understand negativity versus nonnegativity of the rank of $D$
via a dichotomy between acyclic and sourceless partial orientations.
{Theorem~\ref{eff1}} is the culmination of several other results in the
section, primarily {Algorithm~\ref{construct}} which allows one to
determine whether $D$ is linearly equivalent to the divisor associated
to a partial orientation, and {Algorithm~\ref{unfurl}} which lifts the
iterated Dhar's algorithm to the generalized cocycle reversal system.
We begin by reviewing Dhar's algorithm for determining whether a
divisor is $q$-reduced, which is arguably the most important tool for
studying chip-firing.

Each divisor $D$ is equivalent to a unique $q$-reduced divisor $D_{q}$.
This allows one to determine whether $D$ is linearly equivalent to a
effective divisor: simply check whether $D_{q}(q) \geq0$. If $D_{q}(q)
\geq0$, then clearly $D$ is linearly equivalent to an effective
divisor, namely $D_{q}$. On the other hand, if $D_{q}(q) < 0$ then we
can be sure that $D$ is not linearly equivalent to an effective.
Suppose to the contrary that $D$ is equivalent to some $E \geq\vec
{0}$. We can pass from $E$ to $D_{q}$ by repeatedly firing sets of
vertices contained in $V(G) \setminus\{q\}$ without sending any such
vertex into debt, but then we will have that $D_{q}(q) \geq E(q)$, a
contradiction. Here we are implicitly using the fact that $D_{q}$ is
unique as this ensures that when we pass from $E$ to some $q$-reduced
divisor, we obtain $D_{q}$. We refer the reader to Baker and Norine's
original article for further details.

Suppose that $D$ is a divisor such that $D(v) \geq0$ for all $v \neq
q$. \textit{A priori} we would need to check the firing of every
subset of $V(G) \setminus\{ q \}$ to determine whether $D$ is
$q$-reduced, but Dhar's algorithm \cite{dhar1990self} guarantees that
we only need to check a maximal chain of sets. In Dhar's algorithm, we
begin with the set $S = V(G) \setminus\{q\}$. At each step, we
check whether the firing of $S$ would cause some vertex $v$ to be sent
into debt. If so, we remove some such $v$ from $S$ and repeat. Dhar
showed that the algorithm terminates at the empty set if and only if
$D$ is $q$-reduced (recurrent in his setting). This algorithm was
applied by Baker and Norine for proving their fundamental lemma RR1.

This algorithm is often referred to as Dhar's burning algorithm due to
the following interpretation: Suppose we place $D(v)$ firefighters at
each vertex and we start a fire at $q$ which spreads along the edges.
Each firefighter can only stop the fire from coming along one edge at a
time, hence the fire burns through a vertex when it approaches from
more than $D(v)$ directions. Dhar's result can be rephrased as saying
that the fire burns through the graph if and only if $D$ is
$q$-reduced. It is well-known that this process can be applied to
produce bijections between the spanning trees of $G$ and its
$q$-reduced divisors by taking the last edge when burning through a
vertex. Additionally, Dhar's algorithm can be used to produce a
bijection between maximal noneffective $q$-reduced divisors and
$q$-connected acyclic full orientations. We now extend the latter
relationship to (not necessarily maximal) $q$-reduced divisors and
$q$-connected acyclic partial orientations.

\begin{theorem}\label{q-conacyclic}
A divisor $D$ with $D(q)=-1$ is $q$-reduced if and only if $D =
D_{\mathcal{O}
}$, with $\mathcal{O}$ a $q$-connected acyclic partial orientation.
\end{theorem}
\begin{proof}

Let $\mathcal{O}$ be a $q$-connected acyclic partial orientation. We first
claim that $q$ is a source in $\mathcal{O}$. Suppose that this is not
so, and
let $e = (v,q)$ be an edge oriented towards $q$. By assumption, $v$ is
reachable from $q$ by a directed path, and extending this path using
$e$, we see that $q$ belongs to a directed cycle, but this contradicts
the assumption that $\mathcal{O}$ is acyclic. Because $q$ is a source, $D_{q}
= -1$. We next prove that $D_{\mathcal{O}}$ is a $q$-reduced divisor. Let
$v_{0}, \dots, v_{n}$ with $v_{0} = q$ be a total order of the
vertices corresponding to some linear extension of $\mathcal{O}$
viewed as a
poset, i.e. an order of the vertices such that if $v_{j}$ is reachable from $v_{i}$ in
$\mathcal{O}$, then $i < j$. We claim that this sequence of vertices
is an allowable burning sequence from Dhar's algorithm. Suppose this is
not so, and let $v_{i}$ be the first vertex in this sequence such that
if $\{ v_{i} , \dots, v_{n}\} = A$ fires, then $v_{i}$ does not go
into debt. By construction, we know that there does not exists any
oriented edge $(v_{i}, v_{j})$ with $j<i$ nor $(v_{j}, v_{i})$ with
$j>i$. Therefore we conclude that $D_{\mathcal{O}}(v_{i}) \leq\mathrm{outdeg}_{A}(v_{i})-1$, but this contradicts the assumption that firing
the set $A$ does not cause $v_{i}$ to go into debt.

We now prove that if $D$ is a $q$-reduced divisor then $D= D_{\mathcal
{O}}$ for
some $q$-connected acyclic partial orientation. Run the classical
Dhar's algorithm on $D$ to obtain a total order $q < v_{1} < \dots<
v_{n}$ on the vertices. Let $A_{k}$ be the set of vertices which are
less than $v_{k}$. By construction, $|(A_{k}, v_{k})| > D(v_{k}) -1$,
thus we can orient $D(v_{k}) $ of these edges towards $v_{k}$ to obtain
an orientation $\mathcal{O}$ such that $D_{\mathcal{O}} = D$. This
orientation is
acyclic since it has no edges from $v_{j}$ to $v_{i}$ for $v_{i} <
v_{j}$. Suppose that $\mathcal{O}$ is not $q$-connected so that $\bar
{q} \neq
V(G)$, and take $v_{k}$ to be the least vertex in ${\bar{q} }^{c}$. We
know that $D(v_{k}) \geq0$, thus $v_{k}$ must have at least one edge
oriented towards it from $\bar{q}$, a contradiction.
\end{proof}

The proof we have just offered of {Theorem~\ref{q-conacyclic}} is
essentially the same as the classical proofs which show that maximal
non-effective $q$-reduced divisors are in bijection with $q$-connected
acyclic full orientations, e.g. see \cite{benson2010g}. We note that acyclic full orientations are in
bijection with their associated divisors (this can be seen as a special
case of {Corollary~\ref{lem3.1cor}}), but this is clearly not true in general for
acyclic partial orientations.  A very nice investigation of
$q$-connected acyclic partial orientations was conducted by Gessel and
Sagan \cite{gessel1996tutte} who showed that the poset of such partial
orientations admits a natural interval decomposition.

As we will see in this section, for the purposes of divisor theory, it
is interesting to know whether a divisor is associated to an acyclic
partial orientation, but one cares less whether this acyclic partial
orientation is $q$-connected. This allows us to move away from
$q$-reduced divisors and to give the first ``reduced divisor free
proof'' of the Riemann--Roch formula for graphs in section~\ref{dpr&RR}.
We now describe an oriented version of Dhar's algorithm, which allows
us to determine whether a partial orientation is equivalent via edge
pivots to an acyclic partial orientation.

\begin{algorithm} \label{orientedDhar}
Oriented Dhar's algorithm
\end{algorithm}

\noindent{\textbf{Input:}
A partial orientation $\mathcal{O}$ containing a directed cycle and a source.}

\noindent{\textbf{Output:}
A partial orientation $\mathcal{O}'$ with $D_{\mathcal{O}'} =
D_{\mathcal{O}}$, which is either
acyclic or certifies that for every partial orientation $\mathcal
{O}''$ with
$D_{\mathcal{O}''} = D_{\mathcal{O}}$, $\mathcal{O}''$ contains a
directed cycle.}

\bigskip

Initialize by taking $X $ to be the set of sources in $\mathcal{O}$.
At the
beginning of each step, look at the cut $(X,X^{c})$ and perform any
available edge pivots at vertices on the boundary of $X^{c}$ which
bring oriented edges into the cut directed towards $X^{c}$. Afterwards,
for each $v$ on the boundary of $X^{c}$ with no incoming edge contained
in $G[X^{c}]$, add $v$ to $X$. If no such vertex exists, output
$\mathcal{O}'$.

\bigskip
\textbf{Correctness:}
At each step, there are no edges in $(X,X^{c})$ oriented towards $X$.
To prove this, first observe that $X$ satisfies this condition at the
beginning of the algorithm, and note that the vertices added to $X$ at
each step do not cause any such edge to be introduced because any
vertex added does not have an incoming edge in $G[X^{c}]$. It follows
that $X$ will never contain a vertex from a directed cycle: when a
vertex $v$ from a cycle is in the boundary of $X^{c}$, either the cycle
is broken by an edge pivot or $v$ stays in $X^{c}$. Moreover, the
algorithm will never construct directed cycles: if an edge pivot were
to create a cycle, this cycle would intersect $(X,X^{c})$ contradicting
our previous observation. Thus, if the algorithm terminates at $X =
V(G)$, we obtain $\mathcal{O}'$ which is acyclic. If the algorithm terminates
with $X \neq V(G)$, then $\mathcal{O}'$ has a cut saturated towards
$X^{c}$ and
$G[X^{c}]$ is sourceless. By {Corollary~\ref{lem3.1cor}} it suffices to
restrict our attention to partial orientations which are equivalent to
$\mathcal{O}$ by edge pivots. Any other partial orientation $\mathcal
{O}''$ obtained
from $\mathcal{O}'$ by edge pivots will still have $G[X^{c}]$
sourceless, and
thus contain a directed cycle: if we perform a directed walk backwards
along oriented edges in $G[X^{c}]$, this walk will eventually cycle
back on itself.

\bigskip

A vertex $v$ is added to $X$ precisely when firing $X$ causes $v$ to go
into debt. By {Theorem~\ref{q-conacyclic}}, when the input is a partial
orientation with a unique source, {Algorithm~\ref{orientedDhar}} agrees
with Dhar's algorithm at the level of divisors.

If the classical Dhar's algorithm terminates early, we obtain a set
which can be fired without causing any vertex to be sent into debt thus
bringing the divisor closer to being reduced. This leads to what is
known as the iterated Dhar's algorithm. We now extend this method to
the generalized cocycle reversal system. We call our algorithm the
``Unfurling algorithm'' because it consists of breaking cycles in a
partial orientation and spreading edges out towards sinks.

\begin{algorithm} \label{unfurl}
Unfurling algorithm
\end{algorithm}

\noindent{\textbf{Input:}
A partial orientation $\mathcal{O}$ containing a directed cycle and a source.}

\noindent{\textbf{Output:}
A partial orientation $\mathcal{O}'$ equivalent to $\mathcal{O}$ in
the generalized
cocycle reversal system which is either acyclic or sourceless.}

\bigskip

At the $k$th step, run the oriented Dhar's algorithm. If $X = V(G)$,
stop and output $\mathcal{O}$. Otherwise, reverse the consistently
oriented cut
given by the oriented Dhar's algorithm and reset $X$ as the set of
sources (see {Fig.~\ref{2ndunfurlingexample-3}}).

\begin{figure}\label{2ndunfurlingexample-3}
\includegraphics[scale=.6]{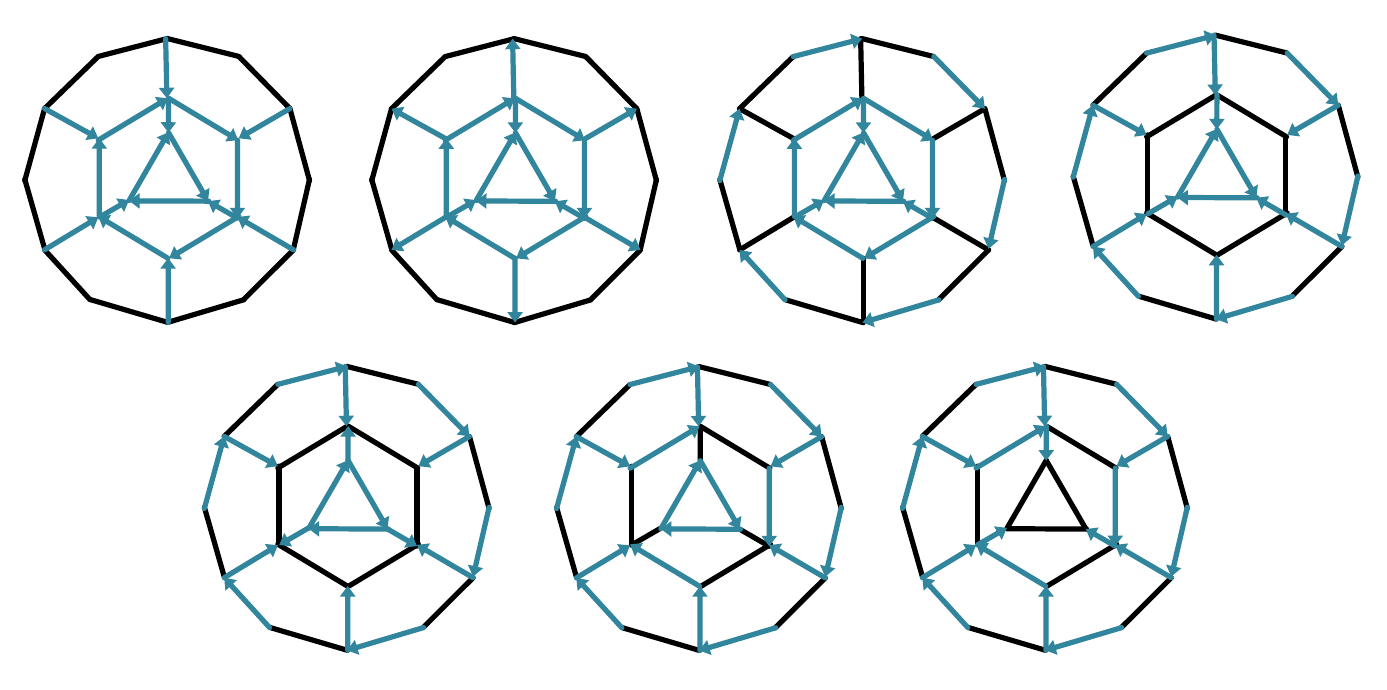}
\caption{The unfurling algorithm applied to the partial orientation on the top
left, terminating with the acyclic partial orientation on the bottom right.} \label{dharexample8}
\end{figure}

\bigskip
\textbf{Correctness:}
The collection of partially orientable divisors linearly equivalent to
$D_{\mathcal{O}}$ is finite, hence the collection of firings which
defines them
is as well, modulo the all 1's vector which generates the kernel of the
Laplacian $\Delta$. In particular, this implies that there exists a
positive integer $c$ such that for all $\mathcal{O}'$ which are
equivalent to
$\mathcal{O}$ in the generalized cycle--cocycle reversal system, for
all $f$
such that $D_{\mathcal{O}'} = D_{\mathcal{O}} - \Delta f$, and for
all $u,v \in V(G)$,
we have that $|f(u)-f(v)| \leq c$.

Let $\mathcal{O}_{k+1}$ be the orientation obtained after the $k$-th
step of
the unfurling algorithm, that is, after the reversal of the $k$-th cut
$(X_{k},X_{k}^{c})$, and let $f_{k+1} = f_{k} + \chi_{X_{k}^{c}}$ so
that $D_{\mathcal{O}_{k}} = D_{\mathcal{O}} - \Delta f_{k}$. Suppose
that the algorithm
failed to terminate. We will demonstrate the existence of vertices $a$
and $b$ such that for every positive integer $c$, there exists $k$ such
that $f_{k}(b)-f_{k}(a)>c$, contradicting our previous observation.

Let $A_{k}$ be the set of sources in $\mathcal{O}_{k}$ and $B_{k}$ be
the set
of vertices belonging to the directed cycles in $\mathcal{O}_{k}$. We claim
that for all $k$, the sets $A_{k+1} \subset A_{k}$ and $B_{k+1} \subset
B_{k}$. The oriented Dhar's algorithm does not create cycles or
sources: The former was proven in the correctness of the oriented
Dhar's algorithm and the latter follows trivially as the oriented
Dhar's algorithm does not change the number of edges oriented towards a
vertex. Additionally, the reversal of the cut $(X_{k},X_{k}^{c})$ does
not create sources or cycles: sources are not created because each
vertex on the boundary of $X_{k}^{c}$ in $\mathcal{O}_{k}$ has at
least one
incoming edge contained in $G[X_{k}^{c}]$, and it is clear that cut
reversals never create cycles. This verifies the claim. For any $a \in
A_{k}$ and $b \in B_{k}$, the value $f_{k}(b)=f_{k-1}(b)+1 = k$ while
$f_{k}(a)=f_{k-1}(a) = 0$. By the assumption that the algorithm does
not terminate, there must exist vertices $a'$ and $b'$ which belong to
$A_{k}$ and $B_{k}$, respectively, for infinitely many $k$, but
$f_{k}(b')-f_{k}(a') = k$ diverges with $k$, a contradiction.

\bigskip

Baker and Norine described the following game of solitaire
\cite[Section 1.5 ]{baker2007riemann}. Suppose you are given a configuration
of chips, can you perform chip-firing moves to bring every vertex out
of debt? There is a natural version of this game for partial
orientations: given a partial orientation, can you find an equivalent
partial orientation which is sourceless? Interestingly, there exists a
dual game in this setting, which does not make much sense in the
context of chip-firing: given a partial orientation, can you find an
equivalent partial orientation which is acyclic? Our unfurling
algorithm gives a winning strategy for at least one of the two games.
We now demonstrate that winning strategies in these games are mutually
exclusive.

\begin{theorem}\label{eff2}
A sourceless partial orientation $\mathcal{O}$ and an acyclic partial
orientation $\mathcal{O}'$ cannot be equivalent in the generalized cocycle
reversal system.
\end{theorem}


\begin{proof}
Let $\mathcal{O}$ be an acyclic partial orientation. We first observe
$\mathcal{O}$
cannot be sourceless. Perform a directed walk backwards starting at an
arbitrary vertex. If this walk were to ever visit a vertex twice, then
$\mathcal{O}$ would contain a cycle, thus it must terminate at a source.

Suppose for the sake of contradiction that $D_{\mathcal{O}} \sim
D_{\mathcal{O}'}$ with
$\mathcal{O}'$ a sourceless partial orientation. Let $f \leq\vec{0}$
be the
vector such that $D_{\mathcal{O}} - \Delta f =D_{\mathcal{O}'}$ and
$S$, the set of
vertices $v$ such that $f(v)=0$. We claim that some $v \in S$ is a
source in $\mathcal{O}'$. Because $\mathcal{O}$ is acyclic, we know
that $\mathcal{O}$
restricted to $S$ is also acyclic, thus there exists some $v \in S$
such that $D_{\mathcal{O}}(v) \leq\mathrm{outdeg}_{S}(v)-1$. It follows that
$D_{\mathcal{O}'}(v) < 0$ since every vertex in $S$ loses at least its
outdegree in the firing of $f$, but this contradicts the assumption
that $\mathcal{O}'$ is sourceless.
\end{proof}

It is easy to see that the previous argument implies the following
stronger statement.

\begin{corollary}\label{r=-1}
Let $\mathcal{O}$ be an acyclic partial orientation. For any $D \sim
D_{\mathcal{O}}$
there exists some $v \in V(G)$ such that $D(v) \leq-1$, i.e.,
$r(D_{\mathcal{O}})=-1$.
\end{corollary}

We would like to be able to use the generalized cycle--cocycle reversal
system for investigating questions about arbitrary divisors of degree
at most $g-1$, but to do so we will need to understand when a divisor
is linearly equivalent to a partially orientable divisor. For
describing our algorithmic solution to this problem, we first introduce
the following modified version of the unfurling algorithm.

\begin{algorithm} \label{mod}
Modified unfurling algorithm
\end{algorithm}

\noindent{\textbf{Input:}
A partial orientation $\mathcal{O}$ and a set of sources $S$ with
$G[S]$ connected.}

\noindent{\textbf{Output:}
A partial orientation $\mathcal{O}' \sim\mathcal{O}$ such that
either}

$(i)$
 $\mathcal{O}'$ has an edge oriented toward some vertex in $S$ or

$(ii)$
 $\mathcal{O}'$ is acyclic which guarantees that for every
$\mathcal{O}'' \sim\mathcal{O}
$, $S$ is a subset of the sources. Moreover, for any $D \sim
D_{\mathcal{O}}$
with $D(s)\geq0$ for some $s \in S$, there exists $v \in V(G)$ such
that $D(v) < -1$.

\bigskip

Initialize with $X_{0}:= S$. At the beginning of each step, look at the
cut $(X_{k},X_{k}^{c})$ and perform any available edge pivots at
vertices on the boundary of $X_{k}^{c}$ which bring oriented edges into
the cut directed towards $X_{k}^{c}$. Afterwards, if there exists some
unoriented edge $(u,v)$ in $(X_{k},X_{k}^{c})$ set $X_{k} := X_{k} \cup
\{ v\}$. Otherwise, $(X_{k},X_{k}^{c})$ is consistently oriented
towards $X_{k}^{c}$ and we reverse this cut. If the cut reversal causes
an edge to be oriented towards a vertex in $S$, we output this
orientation $\mathcal{O}'$ and are in case $(i)$, otherwise set
$X_{k+1} := S$.
If we eventually reach $X_{k} =V(G)$, we output this orientation
$\mathcal{O}'$
and are in case $(ii)$.

\bigskip

We emphasize that when deciding whether to reverse a directed cut
$(X_{k},X_{k}^{c})$, we do not care whether $v$ has any incoming edges
contained in $G[X_{k}^{c}]$. This is what distinguishes the modified
unfurling algorithm from the unfurling algorithm,

\bigskip
\textbf{Correctness:} The termination of the algorithm follows by an
argument similar to the one used in the termination of the unfurling
algorithm. If the algorithm terminates with $X = V(G)$, then the
orientation $\mathcal{O}'$ produced is acyclic by an argument similar
to the
one given for the correctness of the oriented Dhar's algorithm. We next
prove that this acyclic orientation guarantees that for any divisor $D
\sim D_{\mathcal{O}'}$ such that $D(s)\geq0$ for some $s \in S$,
there exists
$v \in V(G)$ such that $D(v) <-1$.

Towards a contradiction, suppose that $D_{\mathcal{O}'} \sim D$ such that
$D(s)\geq0$ for some $s \in S$ and that $D \geq-\vec{1}$. Let $f$ be
such that $D= D_{\mathcal{O}'}-\Delta f$ with $f \leq{\vec{1}}$ and
$Y$ the
non-empty set of vertices such that $f(v) = 1$. It is always possible
to assume that $f$ is of this form by adding the appropriate multiple
of $\vec{1}$ to $f$, and we conclude that for all $v$ in $Y$, $D(v)
\leq D_{\mathcal{O}'}(v) - \mathrm{outdeg}_{Y}(v)$. We first claim that
$S \subset
Y^{c}$. Clearly $S \not\subset Y$, otherwise we would have that $D(v)
\leq-1$ for all $v \in S$ contradicting the assumption that $D(s) \geq
0$. If $S \not\subset Y$ and $S \not\subset Y^{c}$, then by the
connectedness of $G[S]$ there exists some $v \in S$ such that $v$ is in
the boundary of $Y$. The firing $f$ causes $v$ to lose a positive
number of chips and so $D(v) < -1$, contradicting the assumption that
$D \geq-\vec{1}$.

Suppose that the algorithm took $k$ rounds before terminating. Because
$S \subset Y^{c}$, there exists some point at which $Y \subset
X_{k}^{c}$, but there exists $v \in Y$ so that at the next step $v
\notin X_{k}^{c}$. This vertex was added to $X_{k}$ because it was
incident to an unoriented edge in $(X_{k},X_{k}^{c})$ and had no
incoming edges in $G[X_{k}^{c}]$. This implies that $D_{\mathcal{O}'}
(v) <
\mathrm{outdeg}_{X_{k}^{c}}(v)-1 \leq\mathrm{outdeg}_{Y}(v)-1$, thus $D(v)
\leq D_{\mathcal{O}'}(v) - \mathrm{outdeg}_{Y}(v) <-1$, contradicting the
assumption that $D \geq-\vec{1}$.

Finally, it is clear that if $\mathcal{O}' \sim\mathcal{O}''$ some
partial orientation
such that there exists $s \in S$ which is not a source in $\mathcal
{O}''$, then
$D_{\mathcal{O}''}(s)\geq0$ and $D_{\mathcal{O}''} \sim D_{\mathcal
{O}'}$. By the previous
argument, we know that there exists some $v \in V(G)$ such that
$D_{\mathcal{O}
''} <-1$, but this contradicts that $D_{\mathcal{O}''}$ is associated to a
partial orientation.

\bigskip

We now apply our modified unfurling {Algorithm~\ref{mod}} to give an
algorithmic solution to the question of when a divisor of degree at
most $g-1$ is linearly equivalent to a partially orientable divisor.

\begin{algorithm} \label{construct}
Construction of partial orientations
\end{algorithm}
\noindent\textbf{Input:} A divisor $D$ with $\mathrm{deg}(D) \leq g-1$.

\noindent
\textbf{Output:} A divisor $D'\sim D$ and a partial orientation $\mathcal
{O}$ such
that either

$(i)$
 $D' = D_{\mathcal{O}}$ or

$(ii)$ $D' \lneq D_{\mathcal{O}}$ with $\mathcal{O}$ acyclic which
guarantees that $D$
is not linearly equivalent to a partially orientable divisor.

\bigskip

We work with partial orientation-divisor pairs $(\mathcal
{O}_{i},D_{i})$ such
that at each step, $D_{\mathcal{O}_{i}} + D_{i} \sim D$. Initialize
with $(\mathcal{O}
_{0},D_{0}) = (\mathcal{O}', D - D_{\mathcal{O}'})$, where $\mathcal
{O}'$ is an arbitrary
partial orientation. At the $i$th step, let $R_{i}$ be the negative
support of $D_{i} $, $S_{i}$ be the positive support in $D_{i} $, and
$T_{i}$ be the set of vertices incident to an unoriented edge in
$\mathcal{O}
_{i}$. While $D_{i} \neq\vec{0}$, we are in exactly one of the three
following cases:

\bigskip
\noindent
\textbf{Case 1:} The set $S_{i} $ is non-empty and $\mathcal{O}_{i}$
is not a full orientation.

\bigskip

If ${\bar{S}_{i}} \cap T_{i} = \emptyset$ then $G[\bar{S_{i}}]$ is
fully oriented and $(\bar{S_{i}}, \bar{S_{i}}^{c})$ is saturated
towards $\bar{S_{i}}$. Because $\mathcal{O}_{i}$ is not a full orientation,
$\bar{S_{i}}^{c}$ is non-empty. We may reverse the cut $(\bar{S_{i}},
\bar{S_{i}}^{c})$, update $\mathcal{O}_{i}$, and continue. By
induction on the
size of $\bar{S_{i}}^{c}$, eventually $\mathcal{O}_{i}$ will be such that
${\bar{S}_{i}} \cap T_{i} \neq\emptyset$.

If ${\bar{S}_{i}} \cap T_{i} \neq\emptyset$ then there exists a
directed path $P$ from some $s \in S_{i}$ to some vertex $t \in T_{i}$.
Perform a Jacob's ladder cascade along $P$ to cause the initial edge
$e$ incident to $s$ to become unoriented. Orient $e$ towards $s$, set
$D_{i+1} = D_{i} -(s) $, and update $\mathcal{O}_{i+1}$. By induction
on
$\mathrm{deg}^{+}(D_{i})$, this will eventually terminate.

\bigskip
\noindent
\textbf{Case 2:} The set $S_{i}$ is non-empty and $\mathcal{O}_{i}$
is a full orientation.

\bigskip

We must have that $R_{i}$ is nonempty as well, otherwise we would have
that the $\mathrm{deg}(D_{i}) >g$. If ${\bar{S}_{i}} \cap R_{i} \neq
\emptyset$, then there exists a path $P$ from some $s \in S_{i}$ to
some $r \in R_{i}$. Reverse the path $P$, set $D_{i+1} = D_{i} -(s)
+(r) $, and update $\mathcal{O}_{i+1}$. Otherwise, the non-empty cut
$({\bar
{S}_{i}}, {\bar{S}_{i}}^{c})$ is saturated towards ${\bar{S}_{i}}$.
Reverse the cut, update ${\bar{S}_{i}}$, and continue.

\bigskip\noindent
\textbf{Case 3:} The set $S_{i} $ is empty.

\bigskip

We must have that $R_{i}$ is nonempty because $D_{i} \neq{\vec{0}}$.
If there exists some $r \in R_{i}$ with an edge $e$ oriented towards
$r$ in $\mathcal{O}_{i}$, unorient $e$, set $D_{i+1} = D_{i} +(r)$ and update
$\mathcal{O}_{i+1}$. Thus we may assume that $R_{i}$ is a set of
sources in $\mathcal{O}
_{i}$. Let $X$ be a maximal subset of $R_{i}$ such that $G[X]$ is
connected. Apply the modified unfurling {Algorithm~\ref{mod}} to $X$ in
$\mathcal{O}_{i}$ to find an equivalent orientation $\mathcal{O}$ in
the generalized
cocycle reversal system which is either acyclic or has an edge oriented
towards some $r \in R$. In the latter case we may again unorient an
edge pointing towards $r$, set $D_{i+1} := D_{i} +(r)$, update
$\mathcal{O}
_{i+1}$, and continue. In the former case, output $\mathcal
{O}_{i}=\mathcal{O}$ as in $(ii)$.

\bigskip
\textbf{Correctness:}
It is easy to check that for $1\leq i\neq j \leq3$, we can go from
Case $i$ to Case $j$ only if $j >i$. Therefore, termination of the
algorithm follows directly from the termination of each particular
case. The correctness of the algorithm is clear except in Case 3 when
we output $\mathcal{O}$ as in $(ii)$. In this case $S_{i} = \emptyset
$, thus
$D_{i} \leq\vec{0}$ and is supported on the set $R_{i}$. We need to
prove that $\mathcal{O}$ guarantees that $D$ is not linearly
equivalent to a
partially orientable divisor. Suppose that in fact $D \sim D_{\tilde
{\mathcal{O}}}$ for $\tilde{\mathcal{O}}$ some partially orientable
divisor so that
$D_{\mathcal{O}}+ D_{i} \sim D_{\tilde{O}}$. Let $r \in R_{i}$ be in
the set
$X$ from Case 3. We have that $D_{\tilde{\mathcal{O}}} (r) \geq0$,
and by
$(ii)$ of the modified unfurling {Algorithm~\ref{mod}}, we conclude that
there exists some $v \in V(G)$ such that $D_{\tilde{\mathcal{O}}}(v)
<-1$, but
this contradicts the assumption that $D_{\tilde{\mathcal{O}}}$ is
partially orientable.

\bigskip

In section~\ref{MFMC} we will describe a variation of {Algorithm~\ref{construct}}, namely {Algorithm~\ref{2nd}}, which utilizes max-flow
min-cut. As an immediate corollary of {Algorithm~\ref{construct}} we
obtain the following result.

\begin{corollary}[An--Baker--Kuperberg--Shokrieh, Theorem~4.7 \cite{an2013canonical}]
\label{full}
Every divisor $D$ of degree $g-1$ is linearly equivalent to an
orientable divisor.
\end{corollary}

\begin{proof}
Suppose that $D$ is not linearly equivalent to an orientable divisor.
It follows from {Algorithm~\ref{construct}} that $D$ is linearly
equivalent to $D' \lneq D_{\mathcal{O}}$, where $\mathcal{O}$ is an
acyclic partial
orientation. But then $g-1 = \mathrm{deg}(D') < \mathrm{deg}(D_{\mathcal
{O}}) \leq g-1$, a contradiction.
\end{proof}

The following theorem provides a characterization of when a divisor is
linearly equivalent to a partially orientable divisor in terms of its
Baker--Norine rank.

\begin{theorem}\label{partialrank}
A divisor $D$ is linearly equivalent to a partially orientable divisor
$D_{\mathcal{O}}$ if and only if $\mathrm{deg}(D) \leq g-1$ and $r(D
+{\vec{1}})
\geq0$.
\end{theorem}

\begin{proof}
If $D$ is linearly equivalent to a partially orientable divisor
$D_{\mathcal{O}
}$, then $D_{\mathcal{O}} +\vec{1} \geq\vec{0}$ and $\mathrm{deg}(D)
\leq g-1$.
Conversely, suppose that $\mathrm{deg}(D) \leq g-1$ and $D \sim D' \geq-
\vec{1}$. If we apply {Algorithm~\ref{construct}} to $D'$ starting with
the empty orientation, we will always be in Case 1 and the algorithm
will necessarily succeed in producing a partial orientation $\mathcal
{O}$ with
$D_{\mathcal{O}}\sim D' \sim D$.
\end{proof}

The following theorem unifies the main results of this section and
suggests that for understanding the ranks of divisors of degree at most
$g-1$, we need only investigate acyclic and sourceless partial
orientations. In the following section we will refine this result with
{Theorem~\ref{path}} which provides a new interpretation of the
Baker--Norine rank.

\begin{theorem}\label{eff1}
Let $D$ be a divisor with $\mathrm{deg}(D)\leq g-1$, then

(i) $r(D) = -1$ if and only if $D \sim D' \leq D_{\mathcal{O}}$ with
$\mathcal{O}$ an
acyclic partial orientation.

(ii) $r(D) \geq0$ if and only if $D \sim D_{\mathcal{O}}$ with
$\mathcal{O}$ a
sourceless partial orientation.
\end{theorem}

\begin{proof}
$(i)$: Suppose that $D \sim D' \leq D_{\mathcal{O}}$ with $\mathcal
{O}$ an acyclic
partial orientation. By {Corollary~\ref{r=-1}}, we know that
$r(D_{\mathcal{O}})
= -1$, which implies $r(D') = -1$ as $D' \leq D_{\mathcal{O}}$, and
$r(D) =
-1$. To obtain the converse, we take $D$ with $r(D) =-1$ and apply
{Algorithm~\ref{construct}} to $D$. Either we obtain the desired acyclic
partial orientation $\mathcal{O}$, or we find that in $D \sim
D_{\mathcal{O}}$ with $\mathcal{O}
$ not acyclic. Now apply the unfurling {Algorithm~\ref{unfurl}} to
$\mathcal{O}$
to obtain an orientation $\mathcal{O}'$ which is either sourceless or acyclic.
By assumption, $r(D) =-1$, therefore $\mathcal{O}$ is not sourceless,
and thus
must be acyclic. Clearly we have $D_{\mathcal{O}'} \sim D$.

$(ii)$: If $D \sim D_{\mathcal{O}}$ with $\mathcal{O}$ a sourceless partial
orientation, then clearly $r(D) \geq0$. Conversely, take $D$ with
$r(D) \geq0$ and apply {Algorithm~\ref{construct}} to $D$. Because
$r(D) \geq0$ we also know that $r(D + \vec{1}) \geq0$, thus by
{Theorem~\ref{partialrank}} we will output some orientation $\mathcal
{O}$ with
$D_{\mathcal{O}} \sim D$. If $\mathcal{O}$ is sourceless we are done,
otherwise we
apply the unfurling {Algorithm~\ref{unfurl}} which will give us some
other orientation $\mathcal{O}'$ with $D_{\mathcal{O}'} \sim D$ such
that $\mathcal{O}'$ is
either sourceless or acyclic. By {Corollary~\ref{r=-1}}, it is
impossible that $\mathcal{O}'$ is acyclic, hence it is sourceless.
\end{proof}

\section{Directed path reversals and the Riemann--Roch formula}
\label{dpr&RR}

In this section we investigate directed path reversals and their
relationship to Riemann--Roch theory for graphs. Theorem \ref{path} establishes that the Baker--Norine rank
of a divisor associated to a partial orientation is one less than the minimum number of directed paths which need to be
reversed in the generalized cocycle reversal system to produce an
acyclic partial orientation. To prove this characterization, we apply
$q$-connected partial orientations, which generalize the $q$-connected
orientations. We then apply this characterization of rank, together
with results from section~\ref{Dhar}, to give a new proof of the
Riemann--Roch theorem for graphs. Baker and Norine's original argument
proceeds by a formal reduction to statements which they call RR1 and
RR2. We instead employ a variant of this reduction introducing
strengthened versions of RR1 and RR2. While our Strong RR2 is an
immediate consequence of Riemann--Roch, Strong RR1 is not, and appears
to be new to the literature.

\begin{lemma}[An--Baker--Kuperberg--Shokrieh {\cite[Theorem
4.12]{an2013canonical}} and Gioan {\cite[Proposition
4.7]{gioan2007enumerating}}] \label{q-full}
Every full orientation is equivalent in the cocycle reversal system to
a $q$-connected orientation.
\end{lemma}
\begin{proof}
Suppose that $\mathcal{O}$ is a full orientation which is not $q$-connected,
then $\bar{q} \neq V(G)$ and $(\bar{q},\bar{q}^{c})$ is saturated
towards $\bar{q}$. We can reverse this cut and induct on $|\bar{q}^{c}|$.
\end{proof}
In fact, both authors show that this orientation is unique in the
cocycle reversal system and apply this observation to show that the set
of divisors associated to $q$-connected full orientations is equal to
the number of spanning trees of $G$. We will only need the existence
part of their statement. In {Theorem~\ref{qconexists}} we present a
generalization of {Lemma~\ref{q-full}} for partial orientations.

\begin{lemma}[RR1]\label{RR1}
If $r(D)=-1$ then there exists $\nu\in{\mathcal{N}}$ such that
$D\leq\nu$.

\end{lemma}

\begin{proof}

Let $D$ be a divisor with $r(D) =-1$. Suppose that $\mathrm{deg}(D) \geq
g$, and let $D' = D - E$ with $\mathrm{deg}(D') = g-1$ and $E \geq\vec
{0}$. Let $\mathcal{O}$ be an orientation with $D_{\mathcal{O}} \sim
D'$ as guaranteed
by {Corollary~\ref{full}}. By {Lemma~\ref{q-full}}, we may take $\mathcal
{O}$ to
be $q$-connected with $q \in\mathrm{supp}(E)$. It follows that $D \sim
D_{\mathcal{O}}+ E \geq0$ and $r(D) \geq0$, a contradiction. We
conclude that
$\mathrm{deg}(D) \leq g-1$.

By {Theorem~\ref{eff1}} $(i)$, there exists a divisor $D'' \sim D$ such
that $D'' \leq D_{\mathcal{O}}$ where $\mathcal{O}$ is an acyclic
partial orientation.
It is a classical fact, whose proof we now give, that any acyclic
partial orientation can be extended to a full acyclic orientation
$\mathcal{O}
'$. Let $e = (u,v)$ be some unoriented edge in $\mathcal{O}$ and
suppose that
both orientations of $e$ cause a directed cycle to appear. This implies
that there exist directed paths in $\mathcal{O}$ from $u$ to $v$ and
$v$ to
$u$, hence a directed cycle was already present in $\mathcal{O}$, a
contradiction. Alternately, if we view $\mathcal{O}$ as a poset, we
may take a
\textit{linear extension} of $\mathcal{O}$ and orient the remaining
edges in
$G$ according to this order. The divisor $D_{\mathcal{O}'}$ has degree $g-1$
and negative rank by {Theorem~\ref{eff1}} $(i)$.
\end{proof}

We will need the following slight strengthening of RR1.

\begin{corollary}\label{RR1a}
Let $D$ be a divisor with $r(D)=-1$, then for all integers $k$ with
$\mathrm{deg}(D)\leq k \leq g-1$, there exists a divisor $\nu_{k}$ with
$\mathrm{deg}(\nu_{k})=k$, $D \leq\nu_{k}$, and $r(\nu_{k}) = -1$.
\end{corollary}

\begin{proof}
By RR1, there exists $\nu$ with $\mathrm{deg}(\nu)=g-1$, $D \leq\nu$,
and $r(\nu) = -1$. Let $\nu- D = E \geq\vec{0}$, and take $E'$ a
divisor with $\vec{0} \leq E' \leq E$ and $\mathrm{deg}(E')=k-\mathrm{deg}(D)$. We claim that $D + E' = \nu_{k}$ satisfies the conditions of
the theorem. This is clear as $\nu_{k} \leq\nu$ and $r(\nu)=-1$,
thus $r(\nu_{k})=-1$.
\end{proof}

The following is a slight variation on Baker and Norine's
\cite[Lemma 2.7]{baker2007riemann}.

\begin{lemma} \label{formula} Let $D$ be a divisor with $\mathrm{deg}(D)
= k \leq g-1$, then
\begin{eqnarray*}r(D) = \min_{\nu\in\mathcal{N}_{k}} \mathrm{deg}^{+}(D-\nu)-1.
\end{eqnarray*}
\end{lemma}

\begin{proof}

Let $E$ be a divisor of degree $0$ such that $D-E= \nu$ with ${\nu\in
\mathcal{N}_{k}}$ achieving the minimum value of $ \mathrm{deg}^{+}(D-\nu
)-1$. Let $E_{1}, E_{2} \geq{\vec{0}}$ be effective divisors with
disjoint supports such that $E = E_{1} -E_{2}$. We have that $\mathrm{deg}^{+}(D-\nu) =\mathrm{deg}^{+}(E_{1}-E_{2}) = \mathrm{deg}(E_{1})$ and
$D- E_{1} = \nu- E_{2}$ so $r(D-E_{1}) =-1$, which implies $r(D) \leq
\mathrm{deg}^{+}(D-\nu)-1$.

We now show the reverse inequality. Take $E_{1}\geq{\vec{0}}$ with
$r(D) = \mathrm{deg}(E_{1})-1$ and $r(D-E_{1})=-1$.
 By {Corollary~\ref{RR1a}} there exists some effective
 divisor $E_{2}$ such that
$D-E_{1}+E_{2} = \nu_{k}$ for some ${\nu_{k} \in\mathcal{N}_{k}}$.
We claim that $E_{1}$ and $E_{2}$ have disjoint supports. Suppose that
this is not so, and let $v \in V(G)$ be in the support of both $E_{1}$
and $E_{2}$. It follows that $D - (E_{1} - (v)) \leq\nu_{k}$, hence
$r(D - E_{1} + (v)) = -1$ and $r(D) < \mathrm{deg}(E_{1}) - 1$, a
contradiction. It follows that $r(D) = \mathrm{deg}(E_{1}) = \mathrm{deg}^{+}(D-\nu_{k})-1 \geq\mathrm{deg}^{+}(D-\nu)-1 $, where $\nu$
attains the minimum value of the function over all $\nu\in\mathcal{N}_{k}$.
\end{proof}

\begin{lemma}[Strong RR1]
If $\mathrm{deg}(D) =k \leq g-1$ then there exists a divisor $D'$ such that $D \leq D'$, $\mathrm{deg}(D') = g-1$, and $r(D) = r(D') $.
\end{lemma}

\begin{proof}
Let ${\nu\in\mathcal{N}_{k}}$ which achieves the minimum value of
$\mathrm{deg}^{+}(D-\nu)$. By {Lemma~\ref{RR1}}, there exists some $E\geq
0$ such that $\nu+E \in\mathcal{N}$. We claim that $ r(D +E) =
r(D)$. Clearly $r(D+E)\geq r(D)$, and we now establish the reverse
inequality. By {Corollary~\ref{formula}}:
\begin{eqnarray*}
r(D+E) +1 &=& \min_{\nu' \in\mathcal{N}_{g-1}}\mathrm{deg}^{+}(D+E-\nu') \leq\mathrm{deg}^{+}(D+E-(\nu+E)) =
\\
& =& \mathrm{deg}^{+}(D-\nu) = r(D) +1.
\qedhere
\end{eqnarray*}
\end{proof}

\begin{lemma}\label{parqcon}
A partial orientation $\mathcal{O}$ which is either sourceless or has
$q$ as
its unique source is equivalent in the generalized cocycle reversal
system to a $q$-connected partial orientation $\mathcal{O}'$.
\end{lemma}

\begin{proof}
Take $\mathcal{O}$ as in the statement of the Lemma.  If $\mathcal{O}$ is sourceless, let $q$ be an arbitrary vertex.
Suppose that $\bar{q} \neq V(G)$ and there exists a potential edge pivot at a vertex on the boundary
of $\bar{q}^{c}$ which would bring an oriented edge from $G[\bar
{q}^{c}]$ into the cut pointing towards $\bar{q}^{c}$. Performing this
edge pivot would enlarge $\bar{q}$, therefore by induction on $|\bar
{q}^{c}|$, we assume that no such edge pivot is available. Because
every vertex in $\mathcal{O}$, with the possible exception of $q$, has
at least
one incoming edge, we conclude that the cut $(\bar{q},\bar{q}^{c})$
is saturated towards~$\bar{q}$. We can then reverse this cut and again
induct on $|\bar{q}^{c}|$.
\end{proof}

\begin{theorem}\label{qconexists}
A divisor $D$ with $\mathrm{deg}(D) \leq g-1$ is linearly equivalent to
divisor associated to a $q$-connected partial orientation if and only
if $r(D + (q)) \geq0$.
\end{theorem}

\begin{proof}
If $D \sim D_{\mathcal{O}}$ with $\mathcal{O}$ a $q$-connected
partial orientation,
then $D_{\mathcal{O}} + (q) \geq\vec{0}$, and $D +(q)\sim
D_{\mathcal{O}} + (q)$,
thus $r(D + (q)) \geq0$. Now suppose that $r(D + (q)) \geq0$. The
case of $\mathrm{deg} (D) = g-1$ has already been dealt with in \cite{an2013canonical}: By {Lemma~\ref{full}}, $D\sim D_{\mathcal{O}}$ for
$\mathcal{O}$ some
full orientation, and by {Lemma~\ref{q-full}}, $\mathcal{O}$ is
equivalent by
cut reversals to a $q$-connected orientation. Now, suppose that $\mathrm{deg}(D) < g-1$ so that $\mathrm{deg}(D +(q)) \leq g-1$. By {Theorem~\ref{eff1}} $(ii)$, we know that $D+(q) \sim D_{\mathcal{O}}$ for $\mathcal
{O}$ a sourceless
partial orientation. By removing an incoming edge at $q$, we obtain a
partial orientation $\mathcal{O}'$ with $D \sim D_{\mathcal{O}'}$ and
$\mathcal{O}'$ is either
sourceless or has a unique source at $q$. We now apply {Lemma~\ref{parqcon}}.
\end{proof}

We remark that the $q$-rooted spanning trees (also known as \textit
{arborescences}) are precisely the $q$-connected partial orientations
associated to the divisor $-(q)$. Additionally, the $q$-connected
partial orientations associated to $\vec{0}$ are the partial
orientations obtained from $q$-rooted spanning trees by orienting a new
edge towards $q$, i.e., they are the \textit{directed spanning unicycles}.

Any two $q$-connected full orientations which are equivalent in the
cycle--cocycle reversal system are equivalent in the cycle reversal
system, i.e., they have the same associated divisors. This result does
not extend to the setting of partial orientations, as the example in
{Fig.~\ref{q-connectednotunique2}} shows.

\begin{figure}
\includegraphics[scale=.6]{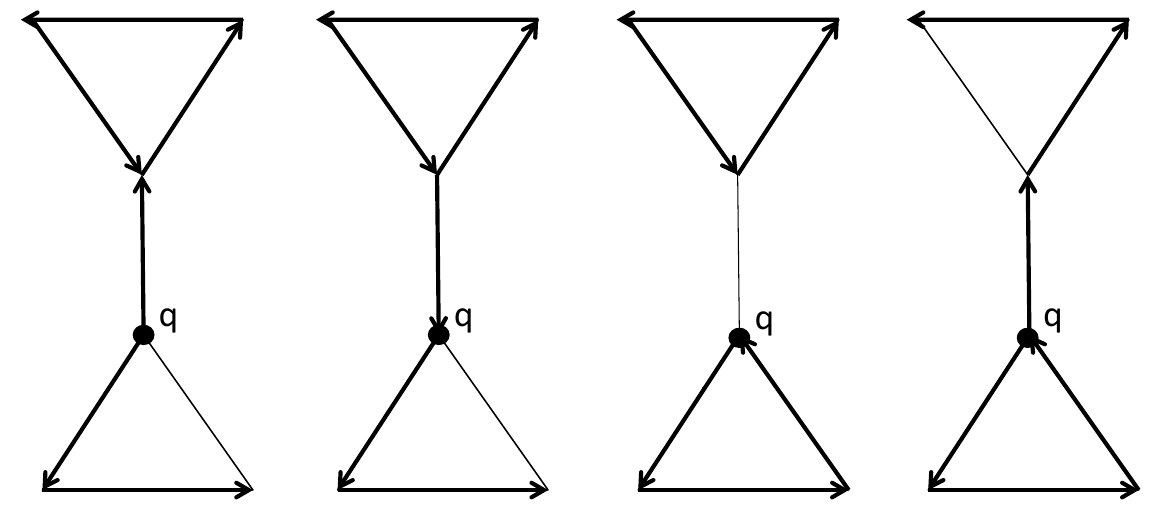}
\caption{A sequence of equivalent partial orientations. The left and right
partial orientations are both $q$-connected, but have different associated
divisors.} \label{q-connectednotunique2}
\end{figure}

The following theorem says that the Baker--Norine rank of a divisor
$D_{\mathcal{O}}$ associated to a partial orientation $\mathcal{O}$
is one less than
the minimum number of directed paths which need to be reversed in the
generalized cycle--cocycle reversal system to produce an acyclic partial
orientation. To make this statement precise, we introduce a helpful
auxiliary graph. Let $\mathcal{G}_{k}$ be a graph with vertex set
\begin{eqnarray*}V(\mathcal{G}_{k}) = \{ [\mathcal{O}]: \mathcal
{O} \mbox{ is  a
 partial  orientation  such that } \mathrm{deg}(D_{\mathcal
{O}}) =
k\}.
\end{eqnarray*}
Two vertices $[\mathcal{O}]$ and $[\mathcal{O}']$ are adjacent in
$\mathcal{G}_{k}$ if
there exist $\mathcal{O}_{1} \in[\mathcal{O}]$ and $\mathcal{O}_{2}
\in[\mathcal{O}']$ such that $\mathcal{O}
_{2}$ is obtained from $\mathcal{O}_{1}$ by reversing some directed
path. Let
\begin{eqnarray*}A = \{ [\mathcal{O}] \in V(\mathcal{G}_{k}): \exists
\mathcal{O}' \in
[\mathcal{O}] \mbox{  with } \mathcal{O}' \mbox{ acyclic}\}
\end{eqnarray*}
and let $d([\mathcal{O}],A)$ be the distance from $[\mathcal{O}]$ to
$A$ in $\mathcal{G}_{k}$.

\begin{theorem} \label{path}
Let $\mathcal{O}$ be a partial orientation with $\mathrm{deg}(D_{\mathcal
{O}}) = k$, then
$r(D_{\mathcal{O}}) = d([\mathcal{O}],A) -\nobreak 1$.
\end{theorem}

\begin{proof}
By {Theorem~\ref{eff1}} $(i)$, $r(D_{\mathcal{O}}) = -1$ if and only if
$[\mathcal{O}]
\in A$, i.e., $d([\mathcal{O}],A) = 0$, thus we assume that
$r(D_{\mathcal{O}}) \geq
0$. Let $d$ be the distance from $[\mathcal{O}]$ to $A$ in
$\mathcal
{G}_{k}$. We will first show that $d-1 \leq r(D_{\mathcal{O}})$. Let
$f_{D_{\mathcal{O}
}} = D_{\mathcal{O}}-\nu$ for $\nu\in{\mathcal{N}}_{k}$ which
achieves the
minimum value of $ \mathrm{deg}^{+}(D_{\mathcal{O}}-\nu)-1$. Recall,
{Lemma~\ref{formula}} states that $r(D_{\mathcal{O}}) = \mathrm{deg}^{+}(D_{\mathcal
{O}}-\nu)-1$.
Because $\mathrm{deg}(f_{D_{\mathcal{O}}}) = 0$, we can write
\begin{eqnarray*}f_{D_{\mathcal{O}}} = \sum_{i=0}^{r(D_{\mathcal
{O}})} (p_{i})-(q_{i}).
\end{eqnarray*}

By {Theorem~\ref{qconexists}} there exists a partial orientation
$\mathcal{O}'$
which is $q_{0}$-connected and $\mathcal{O}' \sim\nobreak \mathcal{O}$. We
can reverse a path
from $q_{0}$ to $p_{0}$ to obtain $\mathcal{O}''$ with $D_{\mathcal
{O}''} = D_{\mathcal{O}'}
+(q_{0}) -(p_{0})$. Proceeding in this way, we arrive at an orientation
$\mathcal{O}'''$ with $D_{\mathcal{O}'''} \sim D - f_{D_{\mathcal
{O}}}$. Therefore, $r(D_{\mathcal{O}
'''}) =-1$ and by {Lemma~\ref{eff1}}, $\mathcal{O}'''$ is equivalent in the
generalized cocycle reversal system to an acyclic partial orientation.
This sequence of partial orientations corresponds to a walk from
$[\mathcal{O}
]$ to $A$ in $V(\mathcal{G}_{k})$ of length $r(D_{\mathcal{O}})+1$, therefore
$d-1 \leq r(D_{\mathcal{O}})$.

Conversely, suppose that we have a sequence of partial orientations
\begin{eqnarray*}\mathcal{O}= \mathcal{O}_{0}, \mathcal{O}_{0}',
\mathcal{O}_{1}, \mathcal{O}_{1}', \dots, \mathcal{O}
_{d}, \mathcal{O}_{d}'
\end{eqnarray*}
where $\mathcal{O}_{i} \sim\mathcal{O}_{i}'$, $\mathcal{O}_{i+1}$
is obtained from $\mathcal{O}_{i}'$
by reversing a directed path from $q_{i}$ to $p_{i}$, and $\mathcal
{O}_{d}'$ is
acyclic. This gives a walk of length $d$ from $[\mathcal{O}]$ to $A$ in
$\mathcal{G}_{k}$. Then $D_{\mathcal{O}} \sim D_{\mathcal{O}_{d}'} +
\sum_{i}^{d}
(p_{i}) - (q_{i})$. It follows that $D_{\mathcal{O}} - \sum_{0}^{d}
(p_{i}) -
(q_{i}) = \nu\in\mathcal{N}_{k}$ and $r(D_{\mathcal{O}}) \leq
\mathrm{deg}^{+}(D_{\mathcal{O}}-\nu)-1 = \mathrm{deg}^{+}(\sum_{0}^{d} (p_{i}) -
(q_{i}))-1 = d-1$.
\end{proof}

{Theorem~\ref{path}} holds in the generalized cocycle reversal system as
well, i.e. where cycle reversals are forbidden;
this follows
from {Corollary~\ref{cocycle}}.  See {Fig.~\ref{pathreversalexample}} for an example of how Theorem \ref{path} can be applied.

\begin{figure}
\includegraphics[scale=.7]{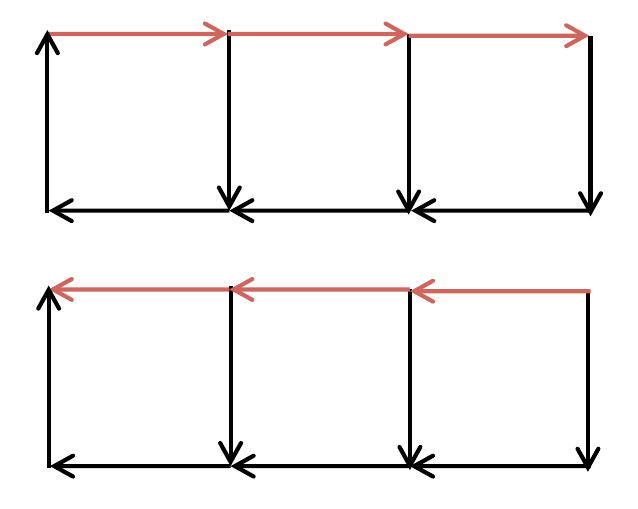}
\caption{A directed path whose reversal produces an acyclic orientation. By
{Theorem~\ref{path}} it follows that the divisor associated to the top
orientation has rank 0.}
\label{pathreversalexample}
\vspace{-6pt}
\end{figure}

One can also describe the Baker--Norine rank of a divisor associated to
a partial orientation as one less than the minimum number of edges
which need to be unoriented in the generalized (cycle--)cocycle reversal
system to produce an acyclic partial orientation. This characterization
follows easily from {Theorem~\ref{eff1}} combined with Baker and
Norine's original description of rank since unorienting an edge in
$\mathcal{O}
$ corresponds to subtracting a chip from $D_{\mathcal{O}}$. A minimum
collection of edges which need to be deleted from an orientation to
destroy all directed cycles is called a min arc feedback set and has
been investigated extensively in the literature. It follows that the
size of a min arc feedback set is a trivial upper bound for the rank of
the divisor associated to a partial orientation. Reed, Robertson,
Seymour, and Thomas \cite{reed1996packing} proved a difficult
Erd\H{o}s--Posa type result, which states that there exists some
function $f:\mathbb{N} \rightarrow\mathbb{N}$ such that
any digraph has either
$k$ edge disjoint directed cycles or there exists a min arc feedback
set of size at most $f(k)$. As an immediate corollary of their work we
have that there exists some function $g: \mathbb{N} \rightarrow
\mathbb{N}$ such that any partial orientation $\mathcal{O}$ either
has $k$
edge disjoint directed cycles or $r(D_{\mathcal{O}}) \leq g(k)$. The author
believes it would be extremely interesting if one were able to apply
ideas from this paper to produce an alternate proof of the Reed,
Robertson, Seymour, and Thomas result. Also see
Perrot and Van Pham \cite{perrot2015feedback}
or Kiss and T\'{o}thm\'{e}r\'{e}sz
\cite{kiss2015chip} where they utilize min arc feedback
sets for investigating complexity questions related to
chip-firing.\vspace{-2pt}

\begin{corollary}[Strong RR2]
If $\mathrm{deg}(D) = g - 1$ then $r(D) = r(K-D)$.
\end{corollary}

\begin{proof}
If $D$ is equivalent to an orientable divisor $D_{\mathcal{O}}$ then
$K-D$ is
equivalent to $K-D_{\mathcal{O}} = D_{\bar{\mathcal{O}}}$, where
${\bar{\mathcal{O}}}$ is the
orientation obtained from $\mathcal{O}$ by reversing the orientation
of every
edge. It is clear by {Theorem~\ref{path}} that $r(D_{\mathcal{O}}) =
r(D_{\bar
{\mathcal{O}}})$ since we may perform mirror operations on the two
orientations.\vspace{-2pt}
\end{proof}

\begin{theorem}[Baker--Norine \cite{baker2007riemann}] For every
divisor $D$ on $G$,\nopagebreak
\begin{eqnarray*}
r(D)-r(K-D) = \mathrm{deg}(D) -g+1.
\end{eqnarray*}
\end{theorem}

\begin{proof}
Either $D$ or $K-D$ has degree at most $g-1$, therefore without loss of
generality, we take $D$ to be a divisor with $\mathrm{deg}(D)\leq g-1$. By
Strong RR1, there exits $E\geq0$ such that $D+E = D'$ with $r(D') =
r(D)$ and $\mathrm{deg}(D')=g-1$. By Strong RR2 we know that $r(D') =
r(K-D')$. To prove the theorem, it suffices to show that
\begin{eqnarray*}r(K-D) - r(K-D') = \mathrm{deg}(K-D) - g+1 = \mathrm{deg}(E).
\end{eqnarray*}
Because $K-D \geq K-D'$, and $\mathrm{deg} (K-D') = g-1$ we know that
\begin{eqnarray*}r(K-D) - r(K-D') \leq\mathrm{deg}(K-D) - g+1 = \mathrm{deg}(E),
\end{eqnarray*}
and for the sake of contradiction, we suppose that
\begin{eqnarray*}r(K-D) - r(K-D') < \mathrm{deg}(E).
\end{eqnarray*}

Let $E'$ be an effective divisors such that $r(K-D-E') = -1$ and
\begin{eqnarray*}\mathrm{deg}(E') = r(K-D)-r(K-D-E')= r(K-D)+1.
\end{eqnarray*}
By RR1, we know that $\mathrm{deg}(K-D-E') \leq g-1$. For an effective
divisor $E''\leq E'$ such that $\mathrm{deg}(K-D-E'') = g-1$, we have that
$\mathrm{deg}(E'') = r(K-D)-r(K-D-E'')$. Note that $\mathrm{deg}(E'') = \mathrm{deg}(E)$, which implies that $r(K-D-E'') < r(K-D')$.

Let $D''$ be the divisor such that $K-D'' = K-D-E''$. We have $D \leq D
+E'' = D''$ so that $r(D) \leq r(D'')$, but
\begin{eqnarray*}r(D'') = r(K-D'') < r(K-D') = r(D') = r(D),
\end{eqnarray*}
a contradiction, thus proving the theorem.
\end{proof}

For a comparison with other proofs of the Riemann--Roch formula for
graphs which appear in the literature, we refer the reader to \cite{amini2012linear,amini2010riemann,baker2007riemann,cori2013riemann,manjunath2012monomials,wilmes2010algebraic}.

\section{Luo's theorem on rank-determining sets}

In this section, we give an alternate proof of Luo's purely topological
characterization of rank-determining sets for metric graphs. Our proof
is based on considerations of acyclic orientations of metric graphs and
directed path reversals. We begin by introducing the necessary notation
and terminology for discussing divisors on metric graphs.

Let $G$ be a connected graph and $w: E(G) \rightarrow\mathbb{R}^{+}$,
a weight function. The \textit{metric graph} $\Gamma$ associated to
$(G,w)$ is the compact connected metric space obtained from $(G,w)$ by
viewing each edge $e$ as isometric to an interval of length $w(e)$. The
\textit{vertices} of $\Gamma$ are the points in $\Gamma$
corresponding to vertices of $G$. We take an \textit{orientation of
$\Gamma$} to be an orientation of the tangent space of $\Gamma$ such
that for any $p \in\Gamma$ and any tangent direction $\tau$ for $p$,
there exists some path emanating from $p$ along $\tau$ of nonzero
length such that the orientation does not change direction. Fix an
orientation $\mathcal{O}(\Gamma)$ arbitrarily. A divisor on $\Gamma$
is a
formal sum of points with integer coefficients and finite support.
Given a piecewise linear function $f$ on $\Gamma$ with integer slopes,
we define $Q(f)$ to be the sum of the incoming slopes minus the
outgoing slopes according to $\mathcal{O}(\Gamma)$, i.e. the
Laplacian applied
to $f$. We say that a divisor $D$ is a \textit{principal divisor} if
it is of the form $Q(f)$, and we say that two divisors are linearly
equivalent if their difference is a principal divisor. We remark that
if all of the edges in $\Gamma$ have length 1, we require all $f$ be smooth in the
interiors of edges, and divisors be supported at vertices, then we
recover the definition of linear equivalence previously given for
discrete graphs using chip-firing.

The definitions of rank, genus, degree, and canonical divisor extend
readily to metric graphs. The Riemann--Roch theorem for metric graphs
was proven independently by Gathmann and Kerber \cite{gathmann2008riemann}, and Mikhalkin and Zharkov \cite{mikhalkin465tropical}. When investigating linear equivalence of
divisors on tropical curves one may forget both the embedding of the
curve and the unbounded rays, thus reducing to the study of metric graphs.

Hladk\'{y}, Kr\'{a}l, and Norine\cite{hladky2013rank} proved that when
computing the rank of a divisor on a metric graph, one need only
consider subtracting chips from the vertices of $\Gamma$, and they
used this result to demonstrate that the rank of a divisor can be
computed in finite time. Luo\cite{luo2011rank} generalized this idea
by defining a set of points $A$ to be \textit{rank-determining} for a
metric graph $\Gamma$ if when computing the rank of any divisor on
$\Gamma$, we only need to subtract chips from points in $A$. A \textit
{special open set} $\mathcal{U}$ is a nonempty, connected, open subset of
$\Gamma$ such that every connected component $X$ of $\Gamma\setminus
\mathcal{U}$ has a boundary point $p$ with $\mathrm{outdeg}_{X}(p)\geq
2$. Luo
introduced a metric version of Dhar's burning algorithm and applied
this technique to obtain the following beautiful {Theorem~\ref{Luo}},
which we now reprove.

Before presenting the proof, we first note a motivating special case:
given an acyclic orientation $\mathcal{O}$ of a metric graph and an
edge $e$ in
which the orientation changes direction, we can perform a directed path
reversal inside of $e$ so that the edge is now oriented towards one of
the two incident vertices without creating a directed cycle. This
follows by a similar argument to the one which was used in our proof of
RR1 for showing that any acyclic partial orientation may be extended
greedily to a acyclic full orientation. By {Lemmas~\ref{Luonegrank} and \ref{Luodegg-1}}, this observation may be converted into a proof that
the vertices of $\Gamma$ are rank-determining, which is \cite[Theorem
3]{hladky2013rank} and \cite[Theorem 1.5]{luo2011rank}. See {Fig.~\ref{VGisrankdet3}}.

\begin{figure}
\includegraphics[scale=.4]{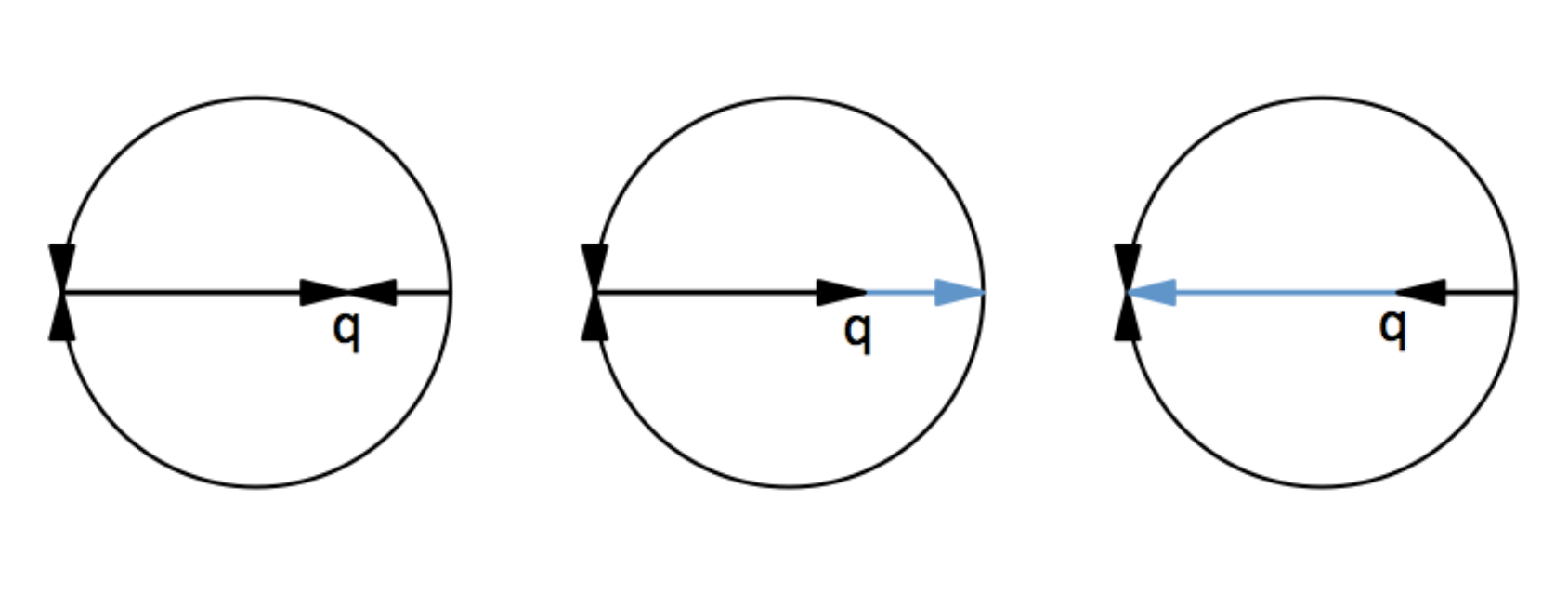}
\caption{A full orientation of a metric graph such that the orientation of the
middle edge changes direction at a point $q$, and two other orientations
obtained by reversing directed paths with one endpoint $q$ and the other
endpoint a vertex. The path reversal in the second orientation causes directed
cycles to appear while the path reversal in the third orientation does not.}
\label{VGisrankdet3}
\vspace{-2pt}
\end{figure}

\begin{lemma}\label{Luonegrank}
A finite subset $A \subset\Gamma$ is rank-determining if and only if
for any divisor $D$ with $r(D)=-1$, and any point $q \in\Gamma$,
there exists a point $a \in A$ such that $r(D+(q)-(a) )=-1$.
\end{lemma}

\begin{proof}
Suppose $A$ is such that for any $D$ with $r(D)=-1$, and any point $q
\in\Gamma$, there exists a point $a \in A$ such that $r(D+(q)-(a)
)=-1$. Let $D'$ be a divisor, and $E$ an effective divisor such that
$r(D'-E) = -1$ and $\mathrm{deg}(E)=r(D')+1$. Let $q \in\mathrm{supp}(E)$.
By assumption, there exists some $a \in A$ such that $r(D-(E - (q)
+(a)) )=-1$. By induction on $\mathrm{deg} (E|_{\Gamma\setminus A})$,
there exists a divisor $E'$ supported on $A$, with $\mathrm{deg}(E) = \mathrm{deg}(E')$ and $r(D-E') = -1$, thus $A$ is rank-determining.

Conversely, suppose that $A$ is rank-determining. Let $D$ be a divisor
with $r(D) = -1$ and $q \in\Gamma$. We known that $r(D+(q)) \leq0$,
therefore there exists some $a \in A$ such that $r(D+(q)-(a)) =-1$.
\end{proof}

\begin{lemma}\label{Luodegg-1}
A finite subset $A \subset\Gamma$ is rank-determining if and only if
for every $\nu\in\mathcal{N}$ and every $q \in\Gamma$, there
exists some $a \in A$ such that $\nu+(q) -(a) \in\mathcal{N}$.
\end{lemma}

\begin{proof}
If $A$ is rank-determining, then {Lemma~\ref{Luonegrank}} says that for
any divisor $D$ with $r(D)=-1$, and any point $q \in\Gamma$, there
exists a point $a \in A$ such that $r(D+(q)-(a) )=-1$. Hence this is
certainly remains true if we restrict $D$ to lie in $\mathcal{N}$.

We now verify the converse. Suppose that $A$ is such that for every
$\nu\in\mathcal{N}$ and every $q \in\Gamma$, there exists some $a
\in A$ such that $\nu+(q) -(a) \in\mathcal{N}$. To verify that $A$
is rank determining it suffices by {Lemma~\ref{Luonegrank}} to prove
that for any divisor $D$ with $r(D)=-1$, and any point $q \in\Gamma$,
there exists a point $a \in A$ such that $r(D+(q)-(a) )=-1$.

By the metric version of RR1, e.g. Mikhalkin and Zharkov \cite[Theorem
7.10]{mikhalkin465tropical}, if a divisor has degree at least $g$, it
has nonnegative rank. Additionally if $D$ is a divisor such that $\mathrm{deg}(D) \leq g-1$ and $r(D)=-1$, then there exists some $\nu\in\mathcal
{N}$ such that $D\leq\nu$. If for every $q \in\Gamma$, there
exists some $a \in A$ such that $r(\nu+(q) -(a))=-1$, then the same
holds for $D$.
\end{proof}

\begin{theorem}[Luo \cite{luo2011rank}, Theorem 3.16] \label{Luo}
A finite subset $A \subset\Gamma$ is rank-determining if and only if
it intersects every special open set $\mathcal{U}$ in $\Gamma$.
\end{theorem}

\begin{proof}
Suppose that $A$ is not rank-determining. By the {Lemmas~\ref{Luonegrank} and \ref{Luodegg-1}}, we may assume that there exists a
divisor $D \in\mathcal{N}$ such that $D +(q) -(a)$ has nonnegative
rank for each $a \in A$. By \cite[Theorem 7.10]{mikhalkin465tropical}
we can take $D$ to be $D_{\mathcal{O}}$ for $\mathcal{O}$ a
$q$-connected acyclic
orientation. The divisor $D_{\mathcal{O}} +(q) -(a)$ has nonnegative
rank and
corresponds to the orientation $\mathcal{O}'$ obtained by reversing a directed
path from $q$ to $a$. Then by \cite[Theorem
7.8]{mikhalkin465tropical}, we know that this orientation cannot be
acyclic, thus we conclude that whenever a path from $q$ to $A$ is
reversed, it causes a directed cycle to appear in the graph.
Equivalently, there exist at least two paths from $q$ to each point of
$A$. Let $\mathcal{U}$ be the set of points which are reachable from
$q$ by a
unique directed path. We claim that $\mathcal{U}$ is a special open
set not
intersecting $A$.

To see that $\mathcal{U}$ is nonempty, notice that the point $q \in
\mathcal{U}$,
otherwise there would be a path from $q$ to itself, implying the
existence of a directed cycle. Every point in $\mathcal{U}$ lies on a
path $P$
from $q$. Moreover, $P \subset\mathcal{U}$, hence by transitivity, ignoring
orientation, $\mathcal{U}$ is connected.

We prove that $\mathcal{U}$ is open by verifying that the complement
of $\mathcal{U}$
is closed. Suppose we have a sequence $S$ of points in $\mathcal{U}^{c}$
converging to some point $p$. There exists some convergent subsequence
$S'$ of $S$ which is contained in an edge $e$ incident to $p$. If we go
far enough along in $S'$ we may assume that all of the points in the
sequence are contained in a consistently oriented segment of $e$. If
this segment is oriented towards $p$, it is clear that $p$ is also
twice reachable from $q$ and thus contained in $\mathcal{U}^{c}$. On
the other
hand, if the edge is oriented away from $p$, the points in our sequence
must be twice reachable from $q$ through $p$, so $p$ is in $\mathcal{U}^{c}$.

Lastly, we show that every connected component $X$ of $\Gamma\setminus
\mathcal{U}$ has a boundary point $p$ with $\mathrm{outdeg}_{X}(p)\geq
2$. Suppose
that there does exist some connected component $X$ of $\mathcal
{U}^{c}$ with
$\mathrm{outdeg}_{X}(p) = 1$ for all boundary points $p$ of $X$. The
restriction of $\mathcal{O}$ to $X$ must also be acyclic, thus it
contains some
source~$s$. This point $s$ cannot be in the interior of~$X$, otherwise
this point would not be reachable from~$q$. Therefore we must have $s =
p$ for $p$ some point on the boundary of~$X$, and $s$ is only reachable
from $q$ along the unique edge $e$ incident to $s$ in $\mathcal{U}$.
But $s \in
\mathcal{U}^{c}$, hence is twice reachable from $q$, therefore so are
all of
the points in $e$ in some neighborhood of $s$, but this contradicts
that these points are in $\mathcal{U}$. This establishes that $A$ is a
rank-determining set.


For demonstrating the converse, we show that given a special open set
$\mathcal{U}$ not intersecting as set $A\subset\Gamma$, we may
construct an
acyclic orientation $\mathcal{O}$ such that $A$ is not
rank-determining for
$D_{\mathcal{O}}$. That is, there exists a point $q \in U$
such that
every $a \in A$ is twice reachable from $q$ in $\mathcal{O}$, which implies
that $r(D_{\mathcal{O}}+(q)-(a))\geq0$ and contradicts {Lemma~\ref{Luodegg-1}}.

Let $q \in\mathcal{U}$ and take a $q$-connected acyclic orientation
of $\mathcal{U}$.
Because $\mathcal{U}$ is connected and open, it follows that $\mathcal
{O}$ will have
sinks at each of the boundary points of $\mathcal{U}$. For any connected
component $X$ of $\mathcal{U}^{c}$ and boundary point $p\in X$ with
$\mathrm{outdeg}_{X}(p) \geq2$, we can construct a $p$-connected acyclic
orientation of $X$. Proceeding in this way for each component $X$, we
obtain a full acyclic orientation $\mathcal{O}$. For $a \in A$, let
$X$ be the
connect component of $\mathcal{U}^{c}$ such that $a \in X$, and $p \in
X$ such
that $\mathcal{O}|_{X}$ is $p$-connected. We know that $p$ is twice reachable
from $q$, hence $a$ is twice reachable from $q$ through $p$. It follows
that the reversal of any path from $q$ to $a$ will cause a directed
cycle to appear in $\Gamma$. This implies that $A$ is not
rank-determining for $D_{\mathcal{O}}+ (q)$ as $D_{\mathcal{O}}\in
\mathcal{N}$, but
$D_{\mathcal{O}}+ (q) -(a) \notin\mathcal{N}$ for any $a \in A$.\looseness=-1
\end{proof}

\section{Max-flow min-cut and divisor theory}\label{MFMC}

In this section we investigate the intimate relationship between
network flows, a topic of fundamental importance in combinatorial
optimization, and the theory of divisors on graphs. We recall that a
\textit{network} $N$ is a directed graph $\vec{G}$ together with a
source vertex $s \in V(\vec{G})$, a sink vertex $t \in V(\vec{G})$,
and a capacity function $c: E(\vec{G}) \rightarrow\mathbb{R}_{\geq
0}$. A \textit{flow} $f$ on $N$ is a function $f : E(\vec{G})
\rightarrow\mathbb{R}_{\geq0}$ such that $f(e) \leq c(e)$ for all $e
\in E(\vec{G})$ and
\begin{eqnarray*}\sum_{e \in E^{+}(v)} f(e) = \sum_{e \in E^{-}(v)} f(e)
\end{eqnarray*}
for all $v \neq s, t$, where $E^{+}(v)$ and $E^{-}(v)$ are the set of
edges pointing towards and away from $v$, respectively. Let $X \subset
V(\vec{G})$ such that $s \in X$. A simple calculation shows that
\begin{eqnarray*}[ll]
\sum_{v \in X}(\sum_{e \in E^{-}(v)}f(e) -\sum_{e
\in E^{+}(v)}f(e)) =
\\
\sum_{e \in\langle X,X^{c}\rangle}f(e) -\sum_{e
\in\langle X^{c},X \rangle}f(e),
\end{eqnarray*}
where $\langle X,X^{c} \rangle$ and $ \langle X^{c},X \rangle$ are
the set of edges in the cut $(X, X^{c})$ directed towards $X^{c}$ and
$X$ respectively. This sum is independent of the choice of $X$, in
particular it is equal to
\begin{eqnarray*}\sum_{e \in E^{-}(s)}f(e) - \sum_{e \in
E^{+}(s)}f(e)= \sum_{e \in E^{+}(t)}f(e)-\sum_{e \in E^{-}(t)}f(e),
\end{eqnarray*}
which we call the flow value from $s$ to $t$ (see {Fig.~\ref{networkexample2}}).

\begin{figure}
\includegraphics[scale=.6]{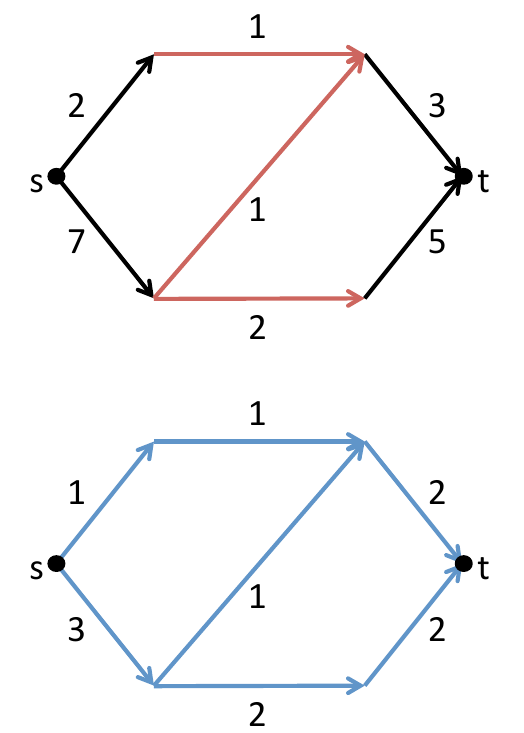}
\caption{Top: A network with source $s$, sink $t$, capacities listed next to
edges, and a minimum cut of size 4 colored red. Bottom: A maximum flow in this
network with flow value 4. Note that the flow along each edge in the minimum cut
is necessarily equal to the capacity of that edge.}
\label{networkexample2}
\end{figure}

One may view a flow as a fluid flow from $s$ to $t$ through a system of
one-way pipes where the capacity of a given edge represents the maximum
rate at which water can travel through the pipe. The flow across any
given cut separating $s$ from $t$ is restricted by the sum of the
capacities of the edges crossing a cut $(X,X^{c})$ towards $t$, which
we denote $c(X)$. The ``max-flow min-cut'' theorem, abbreviated as MFMC,
states that equality is obtained, that is, the greatest flow from $s$
to $t$ is equal to the minimum capacity of a cut separating $s$ from
$t$. This theorem was first proven by
Ford and Fulkerson \cite{ford1956maximal}, and was
independently discovered by Elias,
Feinstein, and Shannon \cite{elias1956note},
 and Kotzig \cite{kotzig1956connectivity}
  the following year. We refer the reader to
Schrijver \cite{schrijver2005history},
 for an interesting account of
the problem's history.

There are two standard methods of proving MFMC, the first is to
demonstrate that a flow of maximum value can be obtained greedily by
so-called augmenting paths which leads to the classical
Ford--Fulkerson
algorithm, and the second is to rephrase the max flow problem as a
linear program and establish MFMC via linear programming duality. We
remark that it has recently been shown that this theorem may also be
viewed as a manifestation of directed Poincar\'{e} duality \cite{ghrist2013topological}.

Momentarily switching gears, we mention the following theorem which
characterizes the collection of orientable divisors on a graph in terms
of Euler characteristics. This result has been rediscovered multiple
times, but appears to originate with S.L. Hakimi \cite{hakimi1965degrees}.
 It might be natural to view his theorem
historically as an extension to arbitrary graphs of Landau's
characterization of score vectors for tournaments \cite{landau1953dominance},
i.e., divisors associated to orientations of the
complete graph, although it seems that Hakimi was unaware of Landau's
result which was presented in a paper on animal behavior a decade earlier.

Recall we define the Euler characteristic of $G[S]$ to be $\chi(S) =
|S|-|E(G[S])|$. Given a divisor $D$ and a non-empty subset $S \subset
V(G)$, we define
\begin{eqnarray*}[cc]
\chi(S,D) = \mathrm{deg}(D|_{S}) +\chi(S)
\\
\chi(G,D) = \mathrm{min}_{S\subset V(G)} \chi(S,D)
\\
{\bar{\chi}(S,D)} = |E(G)| - |E(G[S^{c}])| - |S| -
\mathrm{deg}(D|_{S})
\\
 {\bar{\chi}}(G,D) = \mathrm{min}_{S\subset V(G)} {\bar
{\chi}(S,D)}.
\end{eqnarray*}

\begin{theorem}[Hakimi \cite{hakimi1965degrees}, Felsner \cite{felsner2004lattice},
 An--Baker--Kuperberg--Shokrieh \cite{an2013canonical}] \label{EC}
A divisor $D$ of degree $g-1$ is orientable if and only if $\chi(G,D)
\geq0$.

\end{theorem}

{Theorem~\ref{EC}} states that the orientable divisors on a graph form
the lattice points in a polytope $P$. The \textit{graphical zonotope},
$Z_{G}$ \cite{postnikov2009permutohedra} or \textit{acyclotope} \cite{zaslavsky2013acyclotope}, is the Minkowski sum of the line segments
$[e_{i},e_{j}]$ where $(i,j)$ ranges over all edges in $G$, and $P$ is
obtained by translating $Z_{G}$ by $-\vec{1}$.  We remark that Bartels, Mount, and Welsh \cite{bartels1997win} proved that the partially
orientable divisors on $G$ are the integer points in a polytope which
they call the \textit{win vector polytope}, and the graphical zonotope is a facet of this polytope.  The win vector polytope was rediscovered in an earlier draft of the present article.

There is a ``dual'' formulation of Theorem \ref{EC} which is better suited
for our approach.

\begin{lemma}\label{dualEC}
Let $D$ be a divisor of degree $g-1$, and $S \subset V(G)$, then $\chi
(S,D)\geq0$ if and only if ${\bar{\chi}}(S^{c},D)\geq0$
\end{lemma}

\begin{proof}
This is a straightforward computation.
\end{proof}

\begin{corollary}
If $D$ is a divisor of degree $g-1$, then $\chi(G,D) \geq0$ if and
only if ${\bar{\chi}}(G,D) \geq\nobreak 0$.
\end{corollary}

The following proof originally due to Felsner (and rediscovered
independently by the author) reduces the problem to an application of MFMC.

\begin{proof}[Proof of {Theorem~\ref{EC}}]
If $\mathcal{O}$ is a full orientation, it is clear that $ \chi
(G,D_{\mathcal{O}})
\geq0$. We now establish the converse. Let $D$ be a divisor of degree
$g-1$ satisfying $ \chi(G,D) \geq0$. By {Lemma~\ref{dualEC}} it
follows that ${\bar{\chi}}(G,D) \geq0$. We now demonstrate by
explicit construction that this condition is sufficient to guarantee
the existence of an orientation $\mathcal{O}_{D}$. Let $\mathcal{O}$
be an arbitrary
full orientation and take $\tilde{D} = D-D_{\mathcal{O}}$. Denote the negative
and positive support of $\tilde{D}$ as $S$ and $T$, respectively. Add
two auxiliary vertices $s$ and $t$ with directed edges from $s$ to each
vertex $s' \in\mathrm{supp}(S)$ with capacity $\tilde{D}(s')$ and from
each vertex $t' \in\mathrm{supp}(T)$ to $t$ with capacity $-\tilde
{D}(t')$. Assign each edge in $\mathcal{O}$ capacity 1, and take $N$
be the
corresponding network.

We claim that there is a flow from $s$ to $t$ with flow value $\mathrm{deg}^{+} ({\tilde{D}}) = \mathrm{deg}^{-} ({\tilde{D}})$. By MFMC, to
show that such a flow exists, we need to that show the minimum
capacity of a cut is at least $\mathrm{deg}^{+} (\tilde{D})$. Any $s-t$
cut in $N$ is determined by a set $X \subset\{ V(G) \cup\{s\}\}$. Let
$X \cap T = T_{1}$, $T \setminus T_{1} = T_{2}$, $X \cap S = S_{1}$,
and $S \setminus S_{1} = S_{2}$. The capacity of the cut, $c(X)$ is
equal to $\mathrm{deg}^{-}({\tilde{D}}|_{S_{2}}) + \mathrm{deg}^{+}({\tilde
{D}}|_{T_{1}}) + {\bar{\chi} }(X\setminus\{s\},D_{\mathcal{O}})$.
This is
because ${\bar{\chi} }(X\setminus\{s\},D_{\mathcal{O}})$ counts the number
of edges leaving $X\setminus\{s\}$ in $\mathcal{O}$. We claim that
${\bar
{\chi} }(X\setminus\{s\},D_{\mathcal{O}}) \geq\mathrm{deg}^{-}({\tilde
{D}}|_{S_{1}}) - \mathrm{deg}^{+}({\tilde{D}}|_{T_{1}}) $. Supposing the
claim, we have that $c(X) \geq\mathrm{deg}^{-}({\tilde{D}}|_{S_{2}}) +
\mathrm{deg}^{+}({\tilde{D}}|_{T_{1}}) + \mathrm{deg}^{-}({\tilde
{D}}|_{S_{1}}) - \mathrm{deg}^{+}({\tilde{D}}|_{T_{1})})
= \mathrm{deg}^{-}({\tilde{D}}|_{S_{2}}) + \mathrm{deg}^{-}({\tilde{D}}|_{S_{1}})
= \mathrm{deg}^{-}({\tilde{D}}|_{S})$ as desired.

\begin{figure}
\includegraphics[scale=.9]{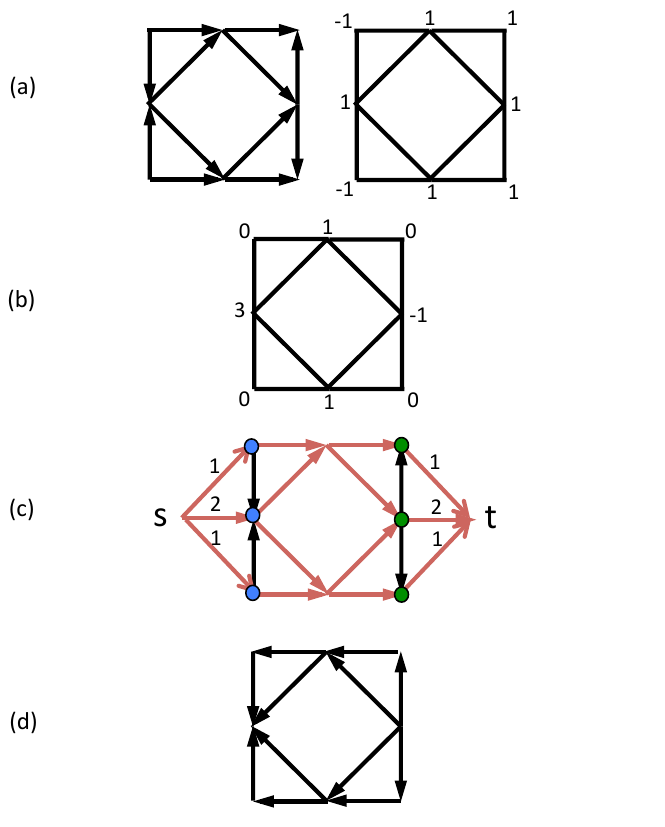}
\caption{(a) Left: An orientation $\mathcal{O}$ of a graph $G$. Right: The
divisor $D_{\mathcal{O}}$ on $G$. (b) A divisor $D$ on $G$. (c) The network $N$
with auxiliary vertices $s$ and $t$. The set $S$ is colored blue, the set $T$ is
colored green, and the additional edges are labeled by their capacities. The
remaining edges have capacity 1. The edges colored red are the support of a
maximal flow $f$. (d) An orientation $\mathcal{O}_{D}$ obtained by reversing the
flow $f$ on $N$ and then restricting to $G$.
(For interpretation of the references to
color in this figure legend, the reader is referred
to the web version of this article.)}
\label{Eulernetex.pdf}
\end{figure}

To prove that ${\bar{\chi} }(X\setminus\{s\},D_{\mathcal{O}}) \geq
\mathrm{deg}^{-}({\tilde{D}}|_{S_{1}}) - \mathrm{deg}^{+}({\tilde{D}}|_{T_{1}})
$ we note that ${\bar{\chi} }(X\setminus\{s\},D_{\mathcal{O}}) =
{\bar{\chi
} }(X\setminus\{s\},D) + \mathrm{deg}^{-}({\tilde{D}}|_{S_{1}}) - \mathrm{deg}^{+}({\tilde{D}}|_{T_{1}})$ and ${\bar{\chi} }(X\setminus\{s\}
,D) \geq0$ by assumption, and the claim follows. Now let $f$ be an max
$s-t$ flow in $N$ with flow value $\mathrm{deg}^{+}({\tilde{D}}|_{S})$.
As is guaranteed by classical proofs of MFMC, we may further take $f$
to be integral. To complete the proof we simply reverse the direction
of each edge in $\mathcal{O}$ in the support of $f$ to obtain a reorientation
of $N$ which when restricted to $G$ gives a desired orientation
$\mathcal{O}
_{D}$. Because $f$ is taken to be integral, we know that an edge is
reversed if and only if the flow along this edge is 1. See {Fig.~\ref{Eulernetex}} for an illustrating example.
\end{proof}

We now demonstrate the converse implication. To the best of the
author's knowledge, this argument has not appeared previously in the literature.

\begin{theorem}\label{equiv}
The max-flow min-cut theorem is equivalent to {Theorem~\ref{EC}}.
\end{theorem}

\begin{proof}
The previous argument shows that max-flow min-cut implies the Euler
characteristic description of orientable divisors {Theorem~\ref{EC}}. We
now demonstrate that {Theorem~\ref{EC}} can be applied in proving MFMC.
It is a classical fact that integer MFMC implies rational MFMC by
scaling, and rational MFMC implies real MFMC by taking limits, thus it
suffices to prove MFMC for networks with integer capacities. Let $N$ be
some network with integer valued capacities which we can view as an
orientation of a multigraph $G$ where the number of parallel edges is
given by the capacities. Suppose that the minimum capacity of a cut
between $s$ and $t$ is of size $k$, and let ${\tilde{D}} = k(t)-k(s)$.
We claim that $D = D_{N}-{\tilde{D}}$ is orientable. By {Theorem~\ref{EC}} and {Lemma~\ref{dualEC}}, it suffices to prove that ${\bar{\chi
}}(G,D)\geq0$. Let $X\subset V(G)$ with $s,t \notin X$. We have that
${\bar{\chi}}(X,D) = {\bar{\chi}}(X,D_{N}) \geq0$. Now take $X
\subset V(G)$ with $s \in X$ and $t \notin X$, and let $c(X)$ be the
capacity associated to this cut. By definition, ${\bar{\chi}}(X,D) +
k = {\bar{\chi}}(X,D_{N}) \geq c(X) $, therefore ${\bar{\chi}}(X,D)
= c(X)-k \geq0$. Finally, we have that ${\bar{\chi}}(X^{c},D) =
{\bar{\chi}}(X^{c},D_{N}) +k \geq0$, and the claim follows.

We next claim that the set of oriented edges $f$ from $\mathcal
{O}_{D}$ which
are oriented differently in $N$ form a flow in $N$ with flow value $k$.
For each vertex $v \in V(G)\setminus\{s,t\}$, reversing the edges in
$f$ preserves the total indegree at $v$, thus the indegree of $v$ in
$f$ equals its outdegree in $f$. For $s$, its outdegree in $f$ minus
its indegree in $f$ is $k$, and for $t$ its indegree in $f$ minus its
outdegree in $f$ is $k$. This proves the claim.
\end{proof}

We leave it to the reader to verify the stronger fact that the flow $f$
in the proof of {Theorem~\ref{equiv}} decomposes as a disjoint union of
$k$ directed paths from $s$ to $t$ along with the possible addition of
some directed cycles.

We remark that if $\mathcal{O}'$ is an integer network, i.e. a full orientation
with distinguished vertices $s$ and $t$, and we wish to find a flow
from $s$ to $t$ of value $k$, we can take $D= k(s)-k(t) + D_{\mathcal{O}'}$.
Applying {Algorithm~\ref{construct}}, we will always be in Case 2, and
we recover the Ford--Fulkerson algorithm. The algorithm produces an
orientation $\mathcal{O}$ such that the set of oriented edges in
$\mathcal{O}$ which
are oriented differently in $\mathcal{O}'$ form a flow of value $k$
from $s$ to $t$.

Let $\Gamma$ be a metric graph. We recall that a \textit{break
divisor} is a divisor of degree $g$ with the property that for all $p \in\Gamma$
there is an injective mapping of chips at $p$ to tangent directions at
$p$, such that if we cut the graph at the specified tangent directions,
we obtain a connected contractable space, i.e., a spanning tree. These
divisors were first introduced in the work of Mikhalkin and Zharkov
\cite{mikhalkin465tropical}, and the following theorem states that
they are precisely the divisors associated to $q$-connected
orientations offset by a chip at $q$. Following \cite{an2013canonical}, we call the divisors associated to $q$-connected
orientations, \textit{$q$-orientable}.

\begin{theorem}[An--Baker--Kuperberg--Shokrieh \cite{an2013canonical}]
\label{break}
A divisor $D$ of degree $g$ is a break divisor if and only if for any
point $q \in\Gamma$, $D-(q)$ is $q$-orientable.
\end{theorem}

An important property of break divisors is that they provide
distinguished representatives for the divisor classes of degree $g$.
Indeed, by {Theorem~\ref{break}}, the set $\{ q\mbox{-orientable divisors}
+(q)\} $ is independent of the choice of $q$. We offer the
following short proof of this fact which does not make use of {Theorem~\ref{break}}. To see that $\{ q\mbox{-orientable divisors} +(q)\} =
\{ p\mbox{-orientable divisors} +(p)\} $ it is equivalent to verify
that $\{ q\mbox{-orientable divisors} +(q)-(p)\}
 =\{ p\mbox{-orientable divisors}\} $. The former set is the collection of
divisors associated to orientations obtained from the $q$-connected
orientations by reversing a path from $q$ to $p$. It is easy to verify
that these are precisely the $p$-connected orientations.

We now describe a simple MFMC based algorithm to obtain the unique
break divisor linearly equivalent to a given divisor of degree $g$.

\begin{algorithm}\label{breakalg}
Efficient method for computing break divisors
\end{algorithm}

\noindent{\textbf{Input:}
A divisor $D$ of degree $g$.}

\noindent{\textbf{Output:}
The break divisor ${\hat{D}} \sim D$.}

\bigskip

Take $q \in V(G)$, and let $D'$ be the divisor of degree $g-1$ with $D'
= D - (q)$. Take $\mathcal{O}$ an arbitrary orientation and construct an
auxiliary network for $D'$ as in the proof of {Theorem~\ref{EC}}: Take
${\tilde{D}} = D' - D_{\mathcal{O}}$ and let ${\tilde{D}}^{+},
{\tilde
{D}}^{-} \geq{\vec{0} }$ be divisors with disjoint supports such that
${\tilde{D}}^{+} - {\tilde{D}}^{-} = {\tilde{D}}$. Let $S$ and $T$
be the support of ${\tilde{D}}^{+}$ and ${\tilde{D}}^{-}$,
respectively. Add two auxiliary vertices $s$ and $t$ with directed
edges from $s$ to each vertex in $S$ and from each vertex in $T$ to
$t$. For each $s' \in S$ and $t' \in T$ we give the edges $(s,s')$ and
$(t,t')$ capacities ${\tilde{D}}^{+}(s')$ and ${\tilde{D}}^{-}(t' )$,
respectively. We can perform any preferred MFMC algorithm which
produces an integral maximum flow $f$ in this network. Reverse all of
the edges in the support of $f$ and update the capacities of the edges
from $s$ to $S$ and $T$ to $t$ to be the residual capacities, i.e., the
original capacities minus the flow value of $f$ on these edges. We've
now obtained a new network such that either when restricted to $G$
gives an orientation $\mathcal{O}'$ with $D_{\mathcal{O}'} = D -
(q)$, or the network
has a directed cut separating $s$ and $t$ which is oriented towards
$s$. We can then reverse this cut and then look for a flow from $s$ to
$t$. Alternating between flow reversals and cut reversals, we
eventually arrive at some orientation $\mathcal{O}''$ such that
$D_{\mathcal{O}''} \sim
D-(q)$. If $\mathcal{O}''$ is not $q$-connected, we may execute
further cut
reversals to obtain a $q$-connected orientation $\mathcal{O}_{q}$. By {Theorem~\ref{break}}, $D_{\mathcal{O}_{q}} +(q) = \hat{D}$ is the break divisor
linearly equivalent to $D$.

\bigskip
\textbf{Correctness:} We first argue that the process terminates. Suppose
that there is no path from $s$ to $t$. We can reverse a cut oriented
towards $s$, and by induction on the size of $\bar{s}$, this will
eventually terminate. Otherwise, there exists a nonzero integral flow exists in
our auxiliary network, which can be reversed.  By induction on ${\tilde{D}}^{+}$ the algorithm terminates.

To see that this process terminates in polynomial time, one can apply
arguments similar to those in \cite{baker2013chip}. This algorithm can
also be sped up by the following preprocessing step. Given $D'$, we can
find $D'' \sim D'$ in polynomial time, which has
bounded size. For example, we can run Algorithm 4 from \cite{baker2013chip}
 to find $D'' \sim D'$ which is
$q$-reduced. It is clear that {Algorithm~\ref{breakalg}} 
applied to $D''$ will terminate in
polynomial time.

\bigskip

By the work of ABKS \cite{an2013canonical}, this method for generating
break divisors can, in principle, be converted into an efficient method
for generating random spanning trees.

Given a divisor $D$ with $\mathrm{deg}(D) \leq g-1$,
{Algorithm~\ref{construct}} provides a method for
constructing a partial orientation
$\mathcal{O}$ with $D_{\mathcal{O}} \sim D$, whenever possible. We
now present an
alternate algorithm which integrates MFMC.

\begin{algorithm}\label{2nd}
A second construction of partial orientations
\end{algorithm}
\noindent
\textbf{Input:} A divisor $D$ with $\mathrm{deg}(D) \leq g-1$.

\noindent
\textbf{Output:} A divisor $D'\sim D$ and a partial orientation $\mathcal
{O}$ such
that either

$(i)$ $D' = D_{\mathcal{O}}$ or

$(ii)$
 $D' \leq D_{\mathcal{O}}$ with $\mathcal{O}$ acyclic which
guarantees that $D$ is
not linearly equivalent to a partially orientable divisor.

\bigskip

Take $D$ with $\mathrm{deg}(D) \leq g-1$, and let $D' = D+E$ with $E\geq
0$ and $\mathrm{deg}(D') = g-1$. First, obtain $\mathcal{O}$ with
$D_{\mathcal{O}} \sim
D'$ by the method described in {Algorithm~\ref{breakalg}}, alternately
reversing flows obtained via some MFMC algorithm, and reversing cuts.
Then perform the modified unfurling {Algorithm~\ref{mod}} to obtain an
orientation with some edge pointed towards a vertex in the support of
$E$. We unorient this edge, subtract a chip from $E$ and repeat.
Eventually we either obtain a partial orientation $\mathcal{O}'$ with
$D_{\mathcal{O}'}
\sim D$ or $\mathcal{O}'$ acyclic and $D_{\mathcal{O}'} \geq D'$ with
$D' \sim D$
which, by the correctness of {Algorithm~\ref{mod}}, guarantees that $D$
is not linearly equivalent to a partially orientable divisor.

\bigskip
\textbf{Correctness:}
This follows directly from {Algorithm~\ref{breakalg}} and the
correctness of {Algorithm~\ref{mod}}.

\bigskip

Motivated by our proof of {Theorem~\ref{EC}}, and by Algorithm \ref{breakalg}, we conclude with the following result. Given a set $X$ and a group $G$, we
say that $X$ is a $G$-\textit{torsor} if $X$ is equipped with a simply
transitive action of $G$.

\begin{theorem}\label{torsor}
The set $\mathrm{Pic}^{g-1}(G)$ is canonically isomorphic as a $\mathrm{Pic}^{0}(G)$-torsor to the collection of equivalence classes in the
cycle--cocycle reversal system acted on by path reversals.
\end{theorem}

\begin{proof}
Let $S$ denote the collection of equivalence classes of full
orientations in the cycle--cocycle reversal system.
By {Corollary~\ref{full}} and {Theorem~\ref{generalizedcc}},
 we can canonically identify
the sets $S$ and $\mathrm{Pic}^{g-1}(G)$. Let $p,q \in V(G)$, $[\mathcal
{O}] \in
S$, and $\mathcal{O}_{q}$ be a $q$-connected orientation in $[\mathcal
{O}]$, which
exists by {Lemma~\ref{q-full}}. The divisor $(q)-(p)$ maps $[\mathcal
{O}]$
to $[\mathcal{O}_{p}]$, where $\mathcal{O}_{p}$ is obtained from
$\mathcal{O}_{q}$ by
reversing the path from $q$ to $p$. This action is well-defined since
$D_{\mathcal{O}_{q}} +(q) - (p) = D_{\mathcal{O}_{p}}$. By linearity,
this map extends
to an action of $\mathrm{Div}^{0}(G)$ on $S$. Moreover, this action
respects linear equivalence, and hence defines an action of $\mathrm{Pic}^{0}(G)$ on $S$.
\end{proof}

\section{Acknowledgements}
Many thanks to Matt Baker for orienting the author towards the study of
orientations and for engaging conversations as well as a careful
reading of an early draft. Additional thanks to Olivier Bernardi from
whose FPSAC lecture we learned of Stefan Felsner's previous work, and
to Sergey Norin for his encouragement. Finally, thanks to the anonymous
referee for many excellent suggestions which helped to greatly improve
the presentation herein.

The author was partially supported by a 
Graduate Student Research Assistantship, by the 
(FP7/2007-2013)/ERC Grant Agreement no. {279558}, and by the 
{Center for Application of Mathematical Principles} at the National Institute of
Mathematical Sciences in South Korea during the Summer 2014 Program on
\textit{Applied Algebraic Geometry}.

%
%
%
%
\bibliography{Riemann}

%
%
\end{document}